\documentclass[12pt,a4paper,oneside]{article}
\usepackage{amsmath,amstext,amssymb,amsopn,amsthm,color}
\usepackage{graphicx}

\usepackage{tikz}
\usetikzlibrary{calc}
\usetikzlibrary{arrows}
\usetikzlibrary{patterns}
\usetikzlibrary{through}
\usetikzlibrary{plotmarks}
\usetikzlibrary{shapes}

\usepackage{ushort}

\reversemarginpar
\definecolor{kb}{rgb}{0.1,0.5,0.1}
\definecolor{mr}{rgb}{0.1,0.2,0.7}
\definecolor{tg}{rgb}{0.7,0.1,0.2}

\theoremstyle{plain}

 \newtheorem{thm}{Theorem}[section] 
\newtheorem{lem}[thm]{Lemma} \newtheorem{prop}[thm]{Proposition}
\newtheorem{cor}[thm]{Corollary}
 \theoremstyle{definition}
\newtheorem{definition}{Definition}

\newtheorem{exmp}{Example} \theoremstyle{remark}

\newtheorem{rem}[thm]{Remark}

\numberwithin{equation}{section}

 \newcommand{\R}{\mathbb{R}}

\newcommand{\RR}{\R} \newcommand{\Rd}{{\R^{d}}}
\newcommand{\Rdz}{{\R^{d}\setminus\{0\}}}
\newcommand{\N}{\mathbb{N}}
 \newcommand{\I}{{\rm I}}
\newcommand{\II}{{\rm II}} 
\renewcommand{\H}{\mathbb{H}} \newcommand{\ind}{{\bf 1}}
\renewcommand{\leq}{\leqslant} \renewcommand{\le}{\leq}
\renewcommand{\geq}{\geqslant} \renewcommand{\ge}{\geq}

\DeclareMathOperator{\dist}{dist}
\DeclareMathOperator{\diam}{diam}

 \def\({\left(} \def\){\right)} \def\[{\left[}
  \def\]{\right]} \def\<{\langle} \def\>{\rangle}

\newcommand{\E}{\mathbb{E}}
\newcommand{\p}{\mathbb{P}}

\ushortCreate*(3)(.46){uline}

\newcommand{\A}{({\bf H})}
\newcommand{\As}{({\bf H^*})}

\newcommand{\WUSC}[3]{\textrm{\rm WUSC}(#1,#2,#3)}
\newcommand{\WLSC}[3]{\textrm{\rm WLSC}(#1,#2,#3)}
\newcommand{\lC}{{\underline{c}}}
\newcommand{\uC}{{\overline{C}}}
\newcommand{\la}{{\underline{\alpha}}}
\newcommand{\ua}{{\overline{\alpha}}}
\newcommand{\lt}{{\underline{\theta}}}
\newcommand{\ut}{{\overline{\theta}}}

\newcommand{\Cf}{C_{1}}

\newcommand{\Ci}{{C_2}}
\newcommand{\Cj}{{C_3}}
\newcommand{\Ck}{{C_4}}
\newcommand{\Cb}{C_5}
\newcommand{\Cc}{C_{6}}
\newcommand{\Cd}{C_{7}}
\newcommand{\Ce}{C_{8}}
\newcommand{\Ca}{C_9}

\newcommand{\Cg}{C_{10}}
\newcommand{\Ch}{C_{11}}
\newcommand{\Cl}{{C_1^*}}
\title{
Dirichlet heat kernel
for unimodal L\'evy processes
\footnotetext{
K. Bogdan was partially supported by NCN grant 2012/07/B/ST1/03356,
T. Grzywny was supported by the Alexander von Humboldt Foundation, M. Ryznar was partially supported by NCN grant 2011/03/B/ST1/00423}}

\author{K. Bogdan, T. Grzywny, M. Ryznar\thanks{corresponding author,  Institute of Mathematics and Computer Science,
  Wroc\l{}aw University of Technology,  ul.~Wyb. Wyspia\'{n}skiego
27, 50-370 Wroc\l{}aw, Poland, michal.ryznar@pwr.wroc.pl, tel. +48 71 320 3155}\\
  Institute of Mathematics and Computer Sciences\\
  Wroc\l{}aw University of Technology, Poland}
 \date{\today}
\begin{document}
\maketitle

\begin{abstract}
We
estimate
the
heat kernel of
the
smooth open set for
the
isotropic unimodal pure-jump L\'evy process with infinite L\'evy measure and weakly scaling L\'evy-Kchintchine exponent.
\end{abstract}

\section{Introduction and preliminaries}
\subsection{Motivation
}\label{sec:mot}

Heat kernels provide direct access to properties of operators with
Dirichlet conditions.
For instance the Green function and the harmonic measure
are expressed by the kernel, cf. \eqref{eq:Gbyp}, \eqref{Ikeda-Watanabe3} below.
We shall estimate the heat kernels of open sets $D\subset \Rd$ with $C^{1,1}$ smoothness of the boundary and nonlocal
translation-invariant integro-differential operators
satisfying the maximum principle and  certain unimodality and scaling conditions.
Such operators
are commonly used to model nonlocal phenomena \cite{ MR2373320, MR2680400,  MR2583992, MR3034006,MR3116931}.
Put differently, we shall study the transition density $p_D(t,x,y)$
of jump-type unimodal L\'evy processes $X$ killed upon leaving $D$ under scaling conditions at infinity for the L\'evy-Kchintchine exponent of $X$.

We recall that precise estimates for the heat
kernel of the Laplacian (and the Brownian motion)
were given for $C^{1,1}$ domains in 2002 by Zhang
\cite{MR1900329}.
In 2006 Siudeja \cite{MR2255353} gave upper bounds for the heat
kernel of the fractional Laplacian (and the isotropic stable L\'evy process) in convex sets.
In 2010 Chen, Kim and Song \cite{MR2677618} gave sharp (two-sided) explicit estimates for the heat kernel  of the fractional Laplacian in bounded $C^{1,1}$ open sets.
Gradual extensions
were then obtained
for generators of
many
subordinate Brownian motions
satisfying scaling conditions \cite{MR2677618, MR2923420, MR2981852, 2013arXiv1303.6449C}, and   for processes with comparable L\'evy measure \cite{2012arXiv1212.3092K}.
We note that subordinate Brownian motions form a proper subset of unimodal L\'evy processes;
in this work we present a
synthetic approach to sharp estimates of $p_D(t,x,y)$ for $C^{1,1}$ open sets $D$
and general unimodal L\'evy processes with scaling.

Rather precise but less explicit bounds of $p_D(t,x,y)$ are also known to hold for Lipschitz sets in a number of situations.
Such bounds were first obtained
for  the Laplacian in 2003 by
Varopoulos
\cite{MR1969798}.
In 2010 the present authors proved
that the following factorization,
\begin{equation}\label{eq:ehks}
    {p_{D}(t, x, y)}\approx P^{x}(\tau_{D}>t)
    P^{y}(\tau_{D}>t)p(t,x,y),
  \end{equation}
holds
for the fractional Laplacian
under a geometric condition  on $x,y\in D$ and $t>0$ for every open $D\subset \Rd$ \cite[Theorem~2]{MR2722789}, see also \cite{MR2602155, MR2722789}.
Here $P^{y}(\tau_{D}>t)$ is the survival probability of the corresponding (isotropic stable L\'evy) process $X$, see \eqref{eq:dsp}, and $p(t,x,y)=p_\Rd(t,x,y)$ is the (free) heat kernel
for $D=\Rd$.
Needless to say, the Dirichlet condition prescribed on $D^c$ for the functions in the domains of  the generator reflects the killing of $X$ when the process first leaves $D$. This accounts for the role played in the study by the first exit time $\tau_D$ of $X$ from $D$.
The comparison \eqref{eq:ehks} is uniform in time and space
for cones, homogeneous Lipschitz domains and exterior $C^{1,1}$ sets, cf.  \cite{MR2722789}, \cite{MR2776619}.
For these sets,
\eqref{eq:ehks} is made rather explicit by approximating the survival probability with superharmonic
functions of $X$ \cite{MR2722789}.

The above  Lipschitz setting of  \cite{MR2722789}, namely the approximate factorization of the heat kernel
and the estimates of the survival probability, are closely
related to the so-called boundary Harnack inequality.
The setting
offers a structured
approach to heat kernel estimates of nonlocal operators.
It is
also relevant
in the Markovian context of
\cite{2013arXiv1303.6449C},  where \eqref{eq:ehks}  serves
as an intermediate step leading  to explicit estimates for $C^{1,1}$ sets.
We therefore owe the reader
an explanation why we postpone the
setting
here and instead use an approach which is tailor-made for $C^{1,1}$ sets. The main reason is better economy and clarity of the presentation when the boundary Harnack principle is replaced by explicit estimates of superharmonic functions, and these are now provided by the preparatory work \cite{BGR2}.
The second main reason is
that the boundary Harnack inequality puts additional constraints on the process $X$,
and these may be circumvented in
the present development.
For instance, the so-called truncated stable L\'evy process is manageable by our approach but cannot be resolved by previous methods because the boundary Harnack inequality fails in this case,
see Example~\ref{ex:tru} in Section~\ref{sec:examples}.

To bound
the heat kernel  $p_D(t,x,y)$ of the unimodal L\'evy process
$X$ and
the $C^{1,1}$ set $D$
we use the estimates of the free transition density $p(t,x,y)$ from
\cite{2014-KB-TG-MR-jfa}
and the estimates 
of superharmonic functions of $X$ at the boundary of $D$ from \cite{BGR2}.
For bounded, exterior, and halfspace-like $C^{1,1}$ open sets
we obtain explicit approximate factorizations of $p_D(t,x,y)$ similar to \eqref{eq:ehks},
along with bounds for survival probability.
The results are given in Theorem~\ref{hkC11_gen}, Theorem~\ref{exterior_lower}  and Theorem~\ref{halfspace-like} below.
Our estimates are {\it sharp}, meaning that the ratio of the upper bound and the lower bound is less than a constant, and they are {\it global}, that is hold with a uniform constant for all $t>0$ and $x,y\in \Rd$.
We focus on the transient case, leaving open some cases of recurrent unimodal L\'evy processes on unbounded subsets of the real line (see \cite{MR2722789} for a comprehensive study of the isotropic stable L\'evy processes, including the recurrent case).

Recall that an {\it exterior} set is the complement of a {\it bounded} set, and a {\it halfspace-like} set is one included between two translates of a halfspace.
We thus cover bounded and some unbounded $C^{1,1}$ sets.
Unbounded sets are especially challenging:
the $C^{1,1}$
condition does not specify their geometry at infinity, whereas
the geometry strongly influences the asymptotics of the heat kernel.
We note that the exterior  $C^{1,1}$ sets
and the halfspace-like sets
were studied
for the fractional Laplacian  in \cite{MR2722789} and \cite{MR2776619}.
The case of the subordinate Brownian motions with global scalings is resolved in \cite{2012arXiv1212.3092K} for the halfspace, and \cite{MR2602155, MR2722789} handle the fractional Laplacian in cones.
Our present
estimates for the heat kernel of exterior sets
in Theorem~\ref{exterior_lower}
are new even for the sum of two independent isotropic stable L\'evy processes. Noteworthy, the comparability constants in the estimates do not change upon dilation of $D$ if the scalings of the L\'evy-Kchintchine exponent of $X$ are global, which is an added bonus of our approach. This is so for the ball and for general exterior open sets, see
Corollary~\ref{hk_kula2} and Corollary~\ref{Exterior_Approx}.
In general we strive to control comparability constants because they may be important in scaling arguments and applications to more general Markov processes.
In passing we also refer the reader to \cite{MR2807275} for
heat kernel estimates of unbounded domains for second-order elliptic differential operators.

Our estimates are generally expressed in terms of $V$,
 the renewal function of the ladder-height process of one-dimensional projections of $X$,
but they could equivalently be  expressed in terms of the more familiar L\'evy-Kchintchine exponent
$\psi$ of $X$, see \eqref{cVh1pg}.
Accordingly, we observe a wide range of power-like asymptotics of heat kernels.
The derivative of $V$ is the {\it  \'eminence grise} of the present project, see also \cite{BGR2}. It is quite delicate to control
$V'$,
but under a mild Harnack-type condition $\A$,
$V'$ only influences the comparability constants, not the {\it structure} of the estimates,
thus allowing for the present generality of  results.

Here is a summary of our main estimates.
We denote by $\delta_D(x)$ the distance of $x\in \Rd$ to $D^c$.
The following comparisons are meant to hold for all $x,y\in \Rd$ and $t>0$, i.e. globally:

If
the L\'evy-Kchintchine exponent $\psi$ of the unimodal L\'evy process $X$ has
lower and upper scalings and $D$ is a {\it bounded} $C^{1,1}$ open set, then
\begin{eqnarray*}
p_{D}(t,x,y)
	&\approx&  \p^x(\tau_D>t/2)\p^y(\tau_D>t/2)p(t\wedge t_0,x,y),
\end{eqnarray*}
and
\begin{equation*}  \p^x\(\tau_{D}>t\)\approx  e^{-\lambda_1 t} \(\frac{V(\delta_D(x))}{\sqrt{t\wedge t_0}}\wedge 1\),
\end{equation*}
where $t_0=V^2(r_0)$, $r_0>0$ is sufficiently small and $-\lambda_1$ is the principal Dirichlet eigenvalue for $D$ and the generator of the semigroup of $X$.
The result is proved in Theorem~\ref{hkC11_gen}.

If
$\psi$  has
{\it global}
lower and upper scalings and $D$ is a $C^{1,1}$ {\it halfspace-like} open set, then
$$ p_{D}(t, x, y)
    \approx \p^x(\tau_D>t
)\p^y(\tau_D>t
)p(t,x,y),$$
and
$$\p^x(\tau_D>t)\approx\frac{V(\delta_D(x))}{\sqrt{t}}\wedge 1.$$
The estimates are proved in Theorem~\ref{halfspace-like}.
The same approximate factorization of $p_D$ holds
under global lower and upper scalings of $\psi$ if $D$ is an {\it exterior} $C^{1,1}$ open set in dimension $d\ge 2$, too, except that in that case we have
    $$  \p^x(\tau_D>t)\approx  \frac{V(\delta_D(x))}{\sqrt{t}\wedge 1}\wedge 1.$$
The result is given in Theorem~\ref{exterior_lower},  Proposition~\ref{Exit_time100} and Corollary~\ref{Exterior_Approx}.
In particular we have  $\p^x(\tau_D>t)\approx \p^x(\tau_D>t/2)$ in the above two cases of unbounded $D$, hence the approximate factorizations in all the three cases above may be considered identical 
in bounded time.
In fact, Theorem~\ref{heatKernelCompl}, Remark~\ref{heatKernelCompl_2balls} and Theorem~\ref{lower_hk_estimate}  below give estimates which essentially resolve the asymptotics of the heat kernels in bounded time and space for every $C^{1,1}$ open set $D$, regardless of the geometry of $D$ at infinity, and they are at the heart of our development.

We note that estimates for the Green function can in principle be obtained by integrating the estimates of the heat kernel against time, cf. \eqref{eq:Gbyp} below and \cite{MR2677618, 2012arXiv1212.3092K}.

Here are comments on possible directions of further research:
Other specific unbounded $C^{1,1}$ open sets, e.g. the parabola-shaped domains \cite{MR2163593} deserve some attention, as they may shed light on the generality of approximate factorizations of  heat kernels.
By a theorem of Courr\`ege, if smooth compactly supported functions are in the domain of the generator of a Markovian semigroup on $\Rd$, then the generator is of L\'evy type
\cite{1998-WHoh-habilitation}.
Therefore one should expect similar estimates of superharmonic functions and heat kernels of L\'evy and Markov processes under two-sided unimodal bounds for the intensity of jumps, cf.
\cite{2013arXiv1303.6449C, 2013arXiv1308.0310K}.
In Remark~\ref{rem:gLc} at the end of the paper
we give more details in the case of L\'evy processes which are isotropic and {\it almost} unimodal.
Lastly, rather optimal isotropic upper bounds of $p(t,x,y)$ for
a class of strongly {\it anisotropic} L\'evy-type operators
were given in \cite{MR2794975}. In the anisotropic setting there is little hope for explicit (two-sided) sharp bounds for $p(t,x,y)$, hence for $p_D(t,x,y)$, but integrable isotropic upper bounds for $p(t,x,y)$ and upper bounds for $p_D(t,x,y)$ at the boundary of $D$ would be of much interest.

The paper is composed as follows. In Section~\ref{sec:prel} we recall the sharp estimates of the free heat kernel from \cite{2014-KB-TG-MR-jfa}. In Section~\ref{sec:DC} we present a general framework for estimating heat kernels of jump  processes and we recall
the estimates of  \cite{BGR2} for the first exit time of unimodal L\'evy processes
from $C^{1,1}$ sets.
The upper bounds for $p_D(t,x,y)$ are given in Section~\ref{sec:UpperBound} and
the lower bounds are given in Section~\ref{sec:LowerBound}. In particular we
propose
techniques based on
structure inequalities \eqref{eq:se} and \eqref{A}, which make our proofs shorter even in comparison with the case of the isotropic stable L\'evy process. We also obtain
a number of auxiliary bounds, which may be interesting on their own.
Our estimates are generally uniform
in bounded time and space,
and if global scaling conditions are satisfied or the set is bounded, then 
the
estimates 
are uniform in the whole of time and space. In Section~\ref{sec:uni} we complement the results of
Section~\ref{sec:UpperBound} and Section~\ref{sec:LowerBound} with some
spectral theory to obtain for bounded $C^{1,1}$ sets sharp heat kernel estimates which are global in time and space.
Since they 
are obtained
rather easily, we invest further attention in unbounded sets, the exterior sets and the halfspace-like sets.
Thus, Section~\ref{sec:ud} focuses on processes with global scaling in unbounded sets, and shows best the strengths of our approach.
In Section~\ref{sec:examples} we discuss specific examples of unimodal L\'evy processes, which can be resolved by our methods. We encourage the reader to inspect the examples when following the
general theory.

\subsection{Estimates for the free process}\label{sec:prel}

Below in the paper we consider the Euclidean space $\R^d$ of arbitrary dimension $d\in \N$.
All the considered sets, functions and measures are tacitly  assumed to be Borel.

We write $f(x)\approx g(x)$ and say $f$ and $g$ are comparable if $f, g\ge 0$ and there is a positive number $C\ge 1$, called {\it comparability constant},
such that $C^{-1}f(x)\le g(x)\le C f(x)$ for all $x$.
We write $C=C(a,\ldots,z)$ to indicate that $C$
may be so chosen to depend only on $a,\ldots,z$. Later on in Remark~\ref{ScAgree} we also make a specific convention regarding the dependence of the constants on $\psi$.
Enumerated
capitalized constants
$C_1$, $\Cj$, \ldots, are meant to be fixed throughout the paper.
Our main motivations for such book-keeping 
is to facilitate scaling arguments and applications to more general Markov processes having variable L\'evy characteristics of a given type.

A
(Borel) measure
on $\Rd$ is called isotropic unimodal, in short: unimodal, if
on $\Rd\setminus \{0\}$ it is absolutely continuous with respect to the Lebesgue measure
and has a (finite) {\it radial nonincreasing} density function. Such measures may have an
atom at the origin.
A L\'evy process $X=(X_t, \,t\ge 0)$ \cite{MR1739520}, is called isotropic unimodal, in short: unimodal, if all of its one-dimensional distributions
$p_t(dx)$ are unimodal.
We will consider jump-type precesses $X$.
To actually define $X$,
recall that L\'evy measure is
any
measure
concentrated on $\Rdz$ such that
\begin{equation*}\label{wml}
\int_\Rd \(|x|^2\wedge 1\)\nu(dx)<\infty.
\end{equation*}
Unimodal pure-jump L\'evy processes are characterized in \cite{MR705619}
by
unimodal L\'evy measures
$\nu(dx)=\nu(x)dx=\nu(|x|)dx$.
After fixing  $\nu$ we denote
$$
 \psi(\xi)=\int_\Rd \left(1- \cos \left<\xi,x\right>\right) \nu(dx),\qquad\xi\in\Rd.$$
Unless explicitly stated otherwise, in what follows we
assume that $\nu$ is an {\it infinite} unimodal L\'evy measure, and
$X$ is the (pure-jump unimodal)
L\'{e}vy process in $\Rd$ given by
$$
\E\,e^{i\left<\xi, X_t\right>}=\int_\Rd e^{i\left<\xi,x\right>}p_t(dx)=e^{-t\psi(\xi)}.
$$
The L\'evy-Kchintchine exponent $\psi$   of $X$ is then unbounded.
Since
$\psi$ is a radial function,  we shall write
$\psi(u)=\psi(x)$, if $u=|x|\ge 0$ and  $x\in \Rd$. Without much notice  the same convention applies to all radial functions.
The L\'evy process $X_t^{(1)}$, i.e. the first coordinate of $X_t$,
has the same function $\psi(u)$. Clearly, $\psi(0)=0$ and $\psi(u)>0$ for $u>0$.
We also note that for $t>0$,
$p_t(dx)$ has no atom at $0$. This is equivalent to infiniteness od $\nu$
\cite[Theorem 30.10]{MR1739520}.
In fact, for $t>0$,
$p_t$ has
density function $p_t(x)$
continuous on $\Rdz$
\cite[Lemma 2.5]{TKMR}.
Furthermore, if
the following Hartman-Wintner condition holds,
\begin{align}\label{eqHW}
&\lim_{|\xi|\to \infty}\psi(\xi)/\ln |\xi|=\infty,
\end{align}
then by Fourier inversion
for each $t>0$, $p_t(dx)$ has smooth density function $p_t(x)$  with integrable derivatives of all orders on $\Rd$ \cite[Lemma~3.1]{MR3010850}. In fact, unimodality yields the following characterization.
\begin{lem}\label{HW}
The density function $p_t(x)$ is bounded for every $t>0$ if and only if \eqref{eqHW} holds.
\end{lem}
\begin{proof}
The
necessity of \eqref{eqHW} follows from
\cite[Proposition~2]{2014-KB-TG-MR-jfa} and \cite[Proposition~4.1]{MR3010850}.
\end{proof}

For $r>0$ we
define Pruitt's function \cite{1981P},
\begin{align}
h(r)&=\int\limits_{\Rd}\left(\frac{|z|^2}{r^2}\wedge 1\right)\nu(dz).\label{def:GKh2}
\end{align}
Note that $0< h(r)<\infty$ and
$h$ is  decreasing.

We also consider the renewal function $V$
   of the (properly normalized) ascending ladder-height process of $X_t^{(1)}$.
The ladder-height process is a subordinator with the Laplace exponent
\begin{equation*}\label{kappa}
 \kappa(\xi)=
\exp\left\{\frac{1}{\pi} \int_0^\infty \frac{ \log {\psi}(\xi\zeta)}{1 + \zeta^2} \, d\zeta\right\}, \quad \xi\ge 0,
\end{equation*}
and $V(x)$ is its potential measure of the half-line $(-\infty,x)$.
Silverstein studied $V$ and $V'$ as $g$ and $\psi$ in \cite[(1.8) and Theorem~2]{MR573292}.
The  Laplace transform of $V$ is
\begin{equation}\label{eq:tLV}
\int_0^\infty V(x)e^{-\xi x}dx=\frac{1}{\xi\kappa(\xi)}, \qquad \xi>0.
\end{equation}
For instance,
$V(x)= x^{\alpha/2}$ for $x\ge 0$, if $\psi(\xi)= |\xi|^\alpha$ \cite[Example~3.7]{MR2453779}.
The definition of $V$ is rather implicit and properties of $V$ are delicate. In particular
 the decay properties of $V'$ are not yet fully understood.
For a detailed discussion of $V$ we refer the reader to
\cite{BGR2} and \cite{MR573292}.
We have $V(x)=0$ for $x\le 0$ and $V(\infty):=\lim_{r\to \infty}V(r)=\infty$. Also, $V$ is
subadditive:
\begin{equation}\label{subad}
 V(x+y)\le V(x)+V(y), \quad x,y \in \R.
\end{equation}
It is known that $V$ is absolutely continuous and harmonic on $(0,\infty)$ for $X_t^1$. Also
$V^\prime$
is a
positive harmonic function for
$X_t^1$ on $(0,\infty)$, hence $V$ is actually (strictly) increasing.
For the so-called complete subordinate Brownian motions \cite{MR2978140} $V'$ is monotone, in fact completely monotone, cf. \cite[Lemma~7.5]{BGR2}. This property was crucial for the development in \cite{MR2923420,2012arXiv1212.3092K}, but in general it fails in the present setting cf. \cite[Remark~9]{BGR2}.

We shall use $V$ and its inverse function $V^{-1}$ in the estimates of heat kernels. In fact, $V$ and $\psi$ may be used interchangeably because of the following lemma.
\begin{lem}\label{ch1V}
The constants in the following comparisons depend only on the dimension,
\begin{equation}\label{cVh1pg}
h(r)\approx \left[V(r)\right]^{-2}
\approx \psi(1/r)
,\qquad r>0.
\end{equation}
\end{lem}
\begin{proof}
The constant in the first comparison depends only on $d$
and the second comparison is absolute; see
\cite[Proposition~2]{2014-KB-TG-MR-jfa} and \cite[Proof of Proposition 2.4]{BGR2}.
\end{proof}
\begin{lem}\label{upper_den} There is a constant $\Cf=\Cf(d)$ such that
\begin{equation}\label{B}p_t(x)\le {\Cf}\frac{t}{ |x|^dV^2(|x|)},\qquad t>0,\; x\in\Rdz.\end{equation}
\end{lem}
\begin{proof}
By \cite[Corollary 7 and Proposition~2]{2014-KB-TG-MR-jfa}, there is
$C=C(d)$ such that
$$p_t(x)\leq C\, \frac{t\,\psi(1/|x|)}{|x|^{d}},\qquad t>0,\; x\in\Rdz.$$
Replacing $\psi(1/|x|)$ with $1/V^2(|x|)$ and using Lemma~\ref{ch1V}, we get the present statement.
\end{proof}
Clearly then, we also have  $\nu(x)\leq \Cf|x|^{-d}V(|x|)^{-2}$, $x\in\Rdz.$

It is rather natural to assume (relative) power-type asymptotics
at infinity for the characteristic exponent $\psi$ of $X$.
To this end we consider
 $\psi$ as a function on $(0,\infty)$.  Let
$\lt\in [0,\infty)$.
We say that
$\psi$ satisfies {the} weak lower scaling condition
at infinity (WLSC) if there are numbers
$\la>0
$
and  $\lC\in(0,1]$,  such that
\begin{equation*}\label{eq:LSC2}
 \psi(\lambda\theta)\ge
\lC\lambda^{\,\la} \psi(\theta)\quad \mbox{for}\quad \lambda\ge 1, \quad
\theta>\lt.
\end{equation*}
In short we write $\psi\in\WLSC{\la}{ \lt}{\lC}$ or $\psi\in {\rm WLSC}$.
If $\psi\in\WLSC{\la}{0}{\lC}$, then we say
that $\psi$ satisfies {\it global} WLSC.
Similarly,
let $\ut\in [0,\infty)$.
The weak upper scaling condition
at infinity (WUSC) means that
there are numbers $\ua <2$
and $\uC{\in [1,\infty)}$ such that
\begin{equation*}\label{eq:USC2}
 \psi(\lambda\theta)\le
\uC\lambda^{\,\ua} \psi(\theta)\quad \mbox{for}\quad \lambda\ge 1, \quad\theta>\ut.
\end{equation*}
In short, $\psi\in\WUSC{\ua}{ \ut}{\uC}$ or $\psi\in{\rm WUSC}$.
{\it Global} WUSC means $\WUSC{\ua}{ 0}{\uC}$.
The reader may find representative examples of characteristic exponents with scaling in Section~\ref{sec:examples} below.

 We call $\la$, $\lt$, $\lC$, $\ua$, $\ut$, $\uC$ the scaling characteristics of $\psi$.
We emphasize that in our setting $\psi\in\WLSC{\la}{ \lt}{\lC}\cap \WUSC{\ua}{ \ut}{\uC}$ entails $0<\la\le \ua<2$.
It may help to recall the connection of the weak scalings to the Matuszewska indices \cite{MR898871}.
Namely, $\psi\in{\rm WLSC}$ if and only if the lower Matuszewska index of $\psi$ is positive,
and $\psi\in {\rm WUSC}$ if and only if the upper Matuszewska index of $\psi$ is smaller than $2$.
Furthermore, $\psi$ satisfies global WLSC if and only if the lower Matuszewska indices of $\psi(\lambda)$ and $1/\psi(1/\lambda)$ are positive,
and $\psi$ satisfies global WUSC if and only if the upper Matuszewska indices of $\psi(\lambda)$ and $1/\psi(1/\lambda)$ are smaller than $2$. The connections are explained in \cite[Remark~2 and Section 4]{2014-KB-TG-MR-jfa}. In what follows we usually skip the word ``weak" when referring to scaling.

Here are further remarks from \cite{2014-KB-TG-MR-jfa}:
We have $\psi\in$WLSC($\la$,$\lt$,$\lC$) if and only if $\psi(\theta)/\theta^\la$
is  comparable to a nondecreasing function on $(\lt,\infty)$, and $\psi\in$WUSC($\ua$,$\ut$,$\uC$) if and only if $\psi(\theta)/\theta^\ua$
is  comparable to a nonincreasing function on $(\ut,\infty)$.
 Scalings ``at zero'' may also be considered and are discussed in \cite[Section~3]{2014-KB-TG-MR-jfa}.
Generally, the lower scaling for large arguments changes to upper scaling for small arguments by taking the reciprocal argument, as in the above discussion of Matuszewska indices for global scalings.
We are thus led to the behavior of $V$ and its inverse function, $V^{-1}$, at zero, cf. \eqref{cVh1pg}.
Namely, let $\psi\in\WLSC{\la}{ \lt}{\lC}$ and $K(\theta)=[V(1/\theta)]^{-2}$, $\theta>0$.
By the proof of Lemma~\ref{ch1V} there is an absolute constant $C\ge 1$ such that $K\in \WLSC{\la}{\lt}{\lC/C}$.
By changing the variables: $\omega=1/\theta$, $\eta=1/\lambda$, the scaling yields
\begin{equation}\label{eq:soV}
V(\eta\omega)\le (C/\lC)^{1/2}\eta^{\la/2}V(\omega),\qquad0<\eta\le 1,\quad 0<\omega<1/\lt\ .
\end{equation}
Furthermore,  $K^{-1}(r)=[V^{-1}(1/\sqrt{r})]^{-1}$.
By \cite[Remark~4]{2014-KB-TG-MR-jfa}, $K^{-1}\in \WUSC{1/\la}{K(\lt)}{(\lC/C)^{-1/\la}}$. By changing the variables: $\omega=1/\sqrt{\theta}$, $\eta=1/\sqrt{\lambda}$, the latter scaling reads as
\begin{align}\label{eq:siVi}
\frac{1}{V^{-1}(\omega)}\ge (\lC/C)^{1/\la}\eta^{2/\la}\frac{1}{V^{-1}(\eta\omega)},\qquad 0<\eta\le 1,\quad 0<\omega\le V(1/\lt).
\end{align}
Since
${1}/{V^{-1}(\omega)}$ is nonincreasing,
\eqref{eq:siVi} offers a complementary doubling-type property.

In the case of $\lt=0$, here and in what follows we may interpret $1/0$ as $\infty$.
\begin{rem}\label{rem:ps}
The thresholds $\lt$, $\ut$ in scalings of $\psi$
may be replaced by $\lt/2$, $\ut/2$ etc. at the expense of constants $\lC$, $\uC$, respectively (see \cite[Section~3]{2014-KB-TG-MR-jfa}). We can also {\it proportionally} extend the range of scalings of $V$ and $V^{-1}$.
 \end{rem}

To conform with \cite{2014-KB-TG-MR-jfa},the scaling conditions  below are only  stated in terms of $\psi$.
The following result elaborates on \eqref{B} when scaling is assumed.
\begin{lem}
\label{densityApprox}If  $\psi\in\WLSC{\la}{\theta}{\lC}$,  then there is  $c=c(d,\la)$
such that for $t<  V^2(\theta^{-1})$,
$$p_t(x)   \le c \min\left\{(\lC)^{-d/\la-1} \left[V^{-1}\left( \sqrt{t}\right)\right]^{{-}d},\frac{t}{V^2(|x|)|x|^d}\right\}.$$
If $\psi\in\WLSC{\la}{\theta}{\lC}\cap\WUSC{\ua}{\theta}{\uC}$, then $C^*=C(d,\la,\ua, \lC, \uC)$
 and $r=r(d,\la,\ua, \lC, \uC)$ exist such that
for all $|x|< r_0:=r/\theta$ and $t<V^2(r_0)$,
$$p_t(x) \ge C^* \min\left\{ \left[V^{-1}\left( \sqrt{t}\right)\right]^{{-}d},\frac{t}{V^2(|x|)|x|^d}\right\} .$$
\end{lem}
\begin{proof}
We replace $\psi$ with $V$ and use Lemma~\ref{ch1V} to reformulate \cite[Theorem~21]{2014-KB-TG-MR-jfa}.
\end{proof}
To clarify, the estimates in Lemma~\ref{densityApprox} hold for all $x\in \Rd$ and $t>0$ if $\theta=0$.
Further,
\begin{equation}\label{eq:rp}
\left[V^{-1}\left( \sqrt{t}\right)\right]^{{-}d} < \frac{t}{V^2(|x|)|x|^d} \quad \mbox{if and only if $t>V^2(|x|)$}.
\end{equation}
It is convenient to assume $\lt=\ut=\theta$ in Lemma~\ref{densityApprox}, and it
entails no essential loss of generality because we can take $\theta=\max\{\lt,\ut\}$
or extend the range of the scalings by using Remark~\ref{rem:ps}.
Conversely, the lower bound in Lemma~\ref{densityApprox} implies the lower and upper scalings of $\psi$, see \cite[Theorem~26]{2014-KB-TG-MR-jfa}, which shows the importance of the 
scaling conditions in the study of unimodal L\'evy processes.
The next result is a variant of \cite[Proposition 19]{2014-KB-TG-MR-jfa}.
\begin{lem}\label{sup_p_t}
If $\psi\in \WLSC{\la}{\lt}{\lC}$, $r>0$ and
$0<t\le rV^2(1/\lt)$, then
\begin{equation}\label{sup-p-t-1}
c_2e^{-c_1 r}\left[V^{-1}\left( \sqrt{t/r}\right)\right]^{{-}d}\leq p_t(0) \le c_3\left( 1+ (\lC r)^{-1-d/\la}\right)
\left[V^{-1}\left( \sqrt{t/r}\right)\right]^{{-}d},
\end{equation}
where $c_1$ is an absolute constant, $c_2=c_2(d)$ and  $c_3=c_3(d,\la)$.
\end{lem}
\begin{proof}
Note that \eqref{eqHW} holds and we have
$$
p_t(0)=(2\pi)^{-d}\int_\Rd e^{-t\psi(\xi)}d\xi, \qquad t>0.
$$
Let $\Psi(s)= 1/V^2(s^{-1})$, $s>0$. By the proof of Lemma~\ref{ch1V},
$c_1^{-1}\Psi(s)\leq \psi(s)\leq c_1\Psi(s)$.
Hence,
$$p_t(0)\geq (2\pi)^{-d}\int_{\{{|x|\le \Psi^{-1}}(r/t)\}} e^{-c_1t\Psi(|x|)}dx\geq (2\pi)^{-d}\frac{\omega_d}{d}\(\Psi^{-1}(r/t)\)^{d}e^{-c_1 r},\quad t>0,$$
where $\omega_d=2\pi^{d/2}/\Gamma(d/2)$ is the surface measure of  the unit sphere in $\Rd$.
Since $\Psi^{-1}(s) = \left(V^{-1}(1/\sqrt s)\right)^{-1}$,
the lower bound in \eqref{sup-p-t-1} obtains.
 If ${0<}t\le r V^2(1/\lt) = r/\Psi(\lt) $, then $t c_1^{-1}\le(rc_1^{-1})/\Psi(\lt)$, and \cite[Lemma 16 {and Remark 6}]{2014-KB-TG-MR-jfa}
with $\epsilon=r c_1^{-1}$ and $t c_1^{-1}$ instead of $t$, yields
$$p_t(0)
\le (2\pi)^{-d} \int_{\Rd}e^{ -t c_1^{-1} \Psi(|x|)}dx   \le  c_4 ( 1+(\lC\, r)^{-1-d/\la} ) \left(
\Psi^{-1}(r/t)\right)^{d}.$$
\end{proof}
\begin{rem}\label{rem:sVi}
Under the assumptions of Lemma~\ref{sup_p_t}, if $0<t\le  C V^2(1/\lt)$, then  $p_t(0)\ge c p_{t/2}(0)$
with constant $c=c(X,C)$. This follows from \eqref{sup-p-t-1}, {\eqref{eq:siVi}},
and Remark~\ref{rem:ps}.
\end{rem}

\begin{definition}
We say that
condition $\A$ holds if for every $r>0$  there is $H_r\geq 1$ such that
\begin{equation*}\label{HR}
V(z)-V(y)\le H_r \,V^\prime(x)(z-y)\quad \text{whenever}\quad 0<x\le y\le z\le5x\leq5r.
\end{equation*}
We say that  $\As$ holds if
$H_\infty=\sup_{r>0}H_r<\infty$.

\end{definition}

We consider $\A$ and $\As$ as
variants of Harnack inequality because $\A$
is
implied by
the following property of $V'$:
\begin{equation*}
\sup_{ y\in[x,5x],\ x\leq r} V'(y)\le H_r \inf_{y\in[x,5x],\ x\leq r}V'(y), \qquad r>0.
\end{equation*}
Both the above conditions control relative growth of $V$. If $\A$ holds, then
we may and do chose  $H_r$
nondecreasing in $r$.
By \cite[Section 7.1]{BGR2}, in each of the following cases, $\A$ holds:
\begin{itemize}
\item[1.] $X$ is a subordinate Brownian motion governed by a special \cite{MR2978140} subordinator. (In this case $V$ is concave so $\As$ holds with $H_\infty=1$.)
\item[2.]  $d\ge 3$ and $\psi$
satisfies WLSC. (If  $d\ge 3$ and $\psi\in \WLSC{\la}{0}{\lC}$, then $\As$ holds.)
\item[3.]\label{case3} $d\ge 1$  and $\psi$
satisfies WLSC and WUSC. (If  $d\ge 1$ and $\psi\in \WLSC{\la}{0}{\lC}\cap \WUSC{\ua}{0}{\uC}$, then $\As$ holds.)
\end{itemize}

We do not know any $V$ failing $\A$, nor
a proof that $\A$ always holds in our setting, which would be interesting to know.
Below approximate factorizations of heat kernels are proved under Case~3, from whence $\A$ follows for all dimensions $d=1,2,\ldots$. However, many auxiliary results of independent interest hold under weaker assumptions, see, e.g., Remark~\ref{rem:d1s} below.

\subsection{Dirichlet condition}\label{sec:DC}

Recall that $d\in \N$.
We let $B(x,r)=\{y\in \Rd: |y-x|<r\}$,
the open ball with center at $x\in \Rd$ and radius $r>0$,
and  $B_r=B(0,r)$.
Recall that by $\omega_d=2\pi^{d/2}/\Gamma(d/2)$ we denote the surface measure of $\partial B_1$, the unit sphere in $\Rd$.
We also
let
$\overline{B(x,r)}^c=\(\overline{B(x,r)}\)^c=\{y\in \Rd: |y-x|> r\}$ and
$\overline{B}^c_r=\overline{B(0,r)}^c$.
For  $a\in \RR$, we consider the upper halfspace $\H_a=\{(x_1,\ldots,x_d)\in \Rd:x_d>a\}$. All other halfspaces are obtained by rotations. The ball, the complement of the ball and the halfspace represent three distinctly different geometries at infinity which are in focus in this paper.

We consider nonempty open set $D\subset \Rd$, its
diameter $\diam(D)=\sup\{|y-x|:\, x,y\in D\}$, and
the distance to its complement:
$$\delta_D(x) =\dist(x,D^c),\qquad x\in \Rd.$$

We say that $
D
$ satisfies the {\it inner ball condition} at scale $r$ if $r>0$ and
for every $Q\in \partial D$ there is ball $B(x\rq{},r)\subset D$ such that $Q\in \partial  B(x\rq{},r)$.
We say $D$ satisfies the {\it outer ball condition} at scale $r$ if $r>0$ and
for every $Q\in \partial D$ there is ball $B(x\rq{}\rq{},r)\subset D^c$ such that $Q\in \partial B(x\rq{}\rq{},r)$.

We say that $D$ is of class $C^{1,1}$ at scale
$r$, if $D$ satisfies the inner and outer ball conditions at the scale $r$. We call $B(x\rq{},r)$ and $B(x\rq{}\rq{},r)$ above the {\it inner}\/ and {\it outer}\/ balls for $D$ at $Q$, respectively.
Estimates of potential-theoretic objects for $C^{1,1}$ sets $D$ often rely on the inclusion $B(x\rq{},r)\subset D\subset
\overline{B(x\rq{}\rq{},r)}^c$ and on explicit calculations for its extreme sides.
If $D$ is $C^{1,1}$ at some positive but unspecified scale (hence also at all smaller scales),
then we simply say $D$ is $C^{1,1}$.
We refer the reader to \cite[Lemma~1]{MR2892584} for more delicate aspects of geometry of $C^{1,1}$ sets.

We are interested in the behavior of the unimodal L\'evy process $X$ as it approaches the complement of
the open set $D$.
We shall use the usual Markovian notation: for $x\in \Rd$ we write $\E^x$ and $\p^x$ for the expectation and distribution of $x+X$, but we use the same symbol $X$ for the resulting process \cite[Chapter~8]{MR1739520}.
We shall also alternatively write $
p_t(y-x)=p(t,x,y)$.
We define
the time of the first exit of $X$ from open set $D\subset \Rd$:
$$\tau_D=\inf\{t>0: \, X_t\notin D\}.$$
The transition
density of the process $X$ {\it killed} upon the first exit from $D$  is defined by
\begin{equation*}
p_D(t,x,y)
=p(t,x,y)-\E^x\left[p(t-\tau_D,X_{\tau_D},y);\tau_D<t\right],
\qquad t>0, \, x,y\in \Rd,
\end{equation*}
see \cite{MR1329992}. We call $p_D$
the heat kernel of $X$ on $D$. The definition is rather implicit, but tractable. For instance, the reader may check that $y\mapsto p_{B_r}(t,0,y)$ is a radial function for all $r,t>0$.
It is well know that $p_D$ satisfies the Chapman-Kolmogorov equations, which yields the following simple connection of the heat kernel and the survival probability.
\begin{lem} \label{UB}
 For all $t>0$ and  $x,
y \in \Rd$, we have
$p_D\(t,x,y\) \le p_{t/2}(0)\,\p^x\(\tau_{D}>t/2\)$
and
$$p_D\(t,x,y\) \le p_{t/2}(0)\;\p^x\(\tau_{D}>\frac{t}{4}\)\p^y\(\tau_{D}>\frac{t}{4}\).  $$
  \end{lem}
\begin{proof} The estimates obtain as follows,
\begin{eqnarray*}p_D\(t,x,y\) &=& \int  p_D\(\frac{t}{2},x,z\)p_D\(\frac{t}{2},z,y\)dz \\
&\le&  \sup_{w,y\in \Rd} p_D\(\frac{t}{2},w,y\) \int  p_D\(\frac{t}{2},x,z\)dz
\le p_{t/2}(0)\p^x\(\tau_{D}>\frac{t}{2}\), \end{eqnarray*}\begin{eqnarray*}
p_D\(t,x,y\) &=& \int \int  p_D\(\frac{t}{4},x,z\) p_D\(\frac{t}{2},z,w\)p_D\(\frac{t}{4},w,y\)dzdw \\
&\le&  \sup_{u,v\in \Rd} p_D\(\frac{t}{2},u,v\) \int  p_D\(\frac{t}{2},x,z\)dz
\int  p_D\(\frac{t}{2},w,y\)dw  \\ &\le&p_{t/2}(0)\;\p^x\(\tau_{D}>\frac{t}{4}\)\p^y\(\tau_{D}>\frac{t}{4}\).
\end{eqnarray*}
\end{proof}

\begin{rem}\label{rem:pDz}
If $\p^x(\tau_D=0)=1$, then $p_D(t,x,y)=0$ for all $y\in \Rd$, $t>0$. The assumption holds for all $x\in D^c$ less a polar set because $X$ is symmetric and has transition density function, see \cite[VI.4.10, VI.4.6, II.3.3]{MR0264757}. If $D$ is a $C^{1,1}$ open set, then the assumption holds for all $x\in D^c$  by familiar arguments of radial symmetry and Blumenthal's zero-one law, see \cite[the proof of Proposition~1.2]{MR1329992}.
\end{rem}
The {\it survival probability} may be expressed via $p_D$:
\begin{align}\label{eq:dsp}
\p^x(\tau_D>t)=\int_\Rd p_D(t,x,y)dy,\qquad t>0,\; x\in \Rd,
\end{align}
and the {\it Green function} of $D$ for $X$ is defined as
\begin{align}\label{eq:Gbyp}
G_D(x,y)&=\int_0^\infty p_D(t,x,y)dt, \qquad x,y\in \Rd.
\end{align}
The {\it expected exit time} is
\begin{equation*}\label{defsD1}
\E^x \tau_D=\int_0^\infty \p^x(\tau_D>t)dt=\int_0^\infty\!\int_\Rd p_D(t,x,y)dydt=\int_\Rd G_D(x,y)dy, \qquad x\in \Rd.
\end{equation*}
If $x \in D$, then the $\p^x$-distribution
of $(\tau_D,X_{\tau_D-},X_{\tau_D})$ restricted to the event $\{X_{\tau_D-}\neq X_{\tau_D}\}$ is given by
the following density function
\cite{MR0142153},
\begin{equation}\label{Ikeda-Watanabe1} (0,\infty)\times D\times
D^c\ni(s, u, z) \mapsto\nu(z -u) p_D(s, x, u).\end{equation}
Integrating against $ds$, $du$ and/or $dz$
gives  marginal distributions.
For instance, if $x\in D$, then
\begin{equation}\label{Ikeda-Watanabe3}
\p^x(X_{\tau_D}\in dz)=\(\int_D G_D(x,u) \nu(z-u) du\)dz,
\end{equation}
on $(\overline{D})^c$ or even on $D^c$ if $\p^x(X_{\tau_D-}\in \partial D)=0$.
Such identities resulting from \eqref{Ikeda-Watanabe1}
are called Ikeda-Watanabe formulae. They enjoy intuitive interpretations in terms of ``occupation time measures'' $p_D(s,x,u)duds$ and $G_D(x,u)du$ and ``intensity of jumps'' $\nu(z-u)dz$, cf. \cite[p.17]{MR2569321}.

The following lemma  is instrumental in estimating the heat kernel $p_{D}$.
This present statement
was preceded by \cite[Theorem~4.2]{MR2231884}, \cite[Lemma~3.2]{MR2255353}, \cite[Lemma~2.2]{MR2677618} and \cite[Lemma~2]{MR2722789}.
\begin{lem}\label{lemppu100}
  Consider disjoint open sets $D_1, D_3\subset D$.
Let $D_2=D\setminus (D_1\cup D_3)$.  If $x\in
  D_1$, $y \in D_3$ and $t>0$, then
	\begin{align*} p_{D}(t, x, y)&\le \p^x(X_{\tau_{D_1}}\in D_2)\sup_{s<t,\, z\in D_2} p(s, z, y)
    +  (t\wedge \E^x \tau_{D_1}) \sup_{u\in D_1,\, z\in D_3}\nu(z-u),
\\ p_{D}(t, x, y)&
    \le \p^x(X_{\tau_{D_1}}\in D_2)\sup_{s<t,\, z\in D_2} p_D(s, z, y)
+ \sup_{u\in D_1,\, z\in D_3}\nu(z-u)\times \nonumber\\
   &\times \left (\p^x(\tau_{D_1}>t/2)\int_0^{t/2} \p^y(\tau_{D}>s) ds +  \p^y(\tau_{D}>t/2)\int_0^{t/2} \p^x(\tau_{D_1}>s)ds\right),\\
p_{D}(t, x, y)&\ge   t\,\p^x(\tau_{D_1}>t)
    \,\p^y(\tau_{D_3}>t)\inf_{u\in D_1,\, z\in D_3}\nu(z-u).
  \end{align*}
\end{lem}
\begin{proof}
  By the strong Markov property, $$p_{D}(t, x,
  y)=\E^x[p_{D}(t-\tau_{D_1}, X_{\tau_{D_1}}, y),\tau_{D_1}<t ].$$
By Remark~\ref{rem:pDz}, this equals
$$
\E^x[p_{D}(t-\tau_{D_1}, X_{\tau_{D_1}}, y),\tau_{D_1}<t,
X_{\tau_{D_1}}\in {D}_2 ] + \E^x[p_{D}(t-\tau_{D_1}, X_{\tau_{D_1}},
y),\tau_{D_1}<t, X_{\tau_{D_1}}\in D_3 ] = \I+\II.
$$
Since
$D_3\subset \overline{D_1}^c$, by (\ref{Ikeda-Watanabe1}) the distribution of $(\tau_{D_1},
X_{\tau_{D_1}})$ at $s>0$ and $z\in D_3$, is given by the density function
\begin{equation*}
  \label{eq:dls}
  f^x(s,z)=\int_{{D_1}}p_{{D_1}}(s, x, u)\nu(z-u)du.
\end{equation*}
Let $m=\inf_{u\in D_1,\, z\in D_3}\nu(z-u)$.
For
$z\in D_3$ we have
$f^x(s,z)
\ge m\p^x(\tau_{D_1}>s)
$, and
\begin{eqnarray*}
  \II&=&\int_0^t\int_{D_3}p_{D}(t-s, z, y)f^x(s,z)dz ds\ge
  m\int_0^t\int_{D_3}p_{D}(t-s, z, y)\p^x(\tau_{D_1}>s)dz ds\\&\ge&
  m\p^x(\tau_{D_1}>t)
  \int_0^t\int_{D_3}p_{D_3}(t-s, z, y)dz ds
  =
  m\ \p^x(\tau_{D_1}>t)
  \int_0^t\p^y(\tau_{D_3}>s)ds,
\end{eqnarray*}
hence
the lower bound.
For the upper bounds we let $M=\sup_{u\in D_1,\, z\in D_3}\nu(z-u)$, obtaining
\begin{eqnarray}
  \II&\le&
  M\int_0^t\int_{D_3}p_{D}(t-s, z, y)\p^x(\tau_{D_1}>s)dz ds\nonumber\\
  &\le&
  M\int_0^t \p^x(\tau_{D_1}>s)  P^y(\tau_{D}>t-s) ds \label{eq:ob2}
\\
	&\le&
M \left (  \p^x(\tau_{D_1}>t/2)\int_0^{t/2} \p^y(\tau_{D}>s) ds +\p^y(\tau_{D}>t/2)\int_0^{t/2} \p^x(\tau_{D_1}>s)ds\right).
\nonumber
\end{eqnarray}
This, \eqref{eq:ob2}, and the inequality
$\I
\le \p^x(X_{\tau_{D_1}}\in D_2)\sup_{s<t,\, z\in D_2} p_D(s, z, y)$, finish the proof.
\end{proof}

Similar arguments
provide the following relationship, which will be useful later on.
\begin{lem}\label{HKLB} For all $t>0$ and $y\in \Rd$,
\begin{equation*}
p_t(y)\ge   4^{-d}t\,\nu(y)\[\p^0(\tau_{B_{|y|/2}}>t)\]^2.\end{equation*}
\end{lem}
\begin{proof}
We use the notation from the previous lemma. Let $y\neq 0$, $D_1=B(0,|y|/2), D_3= B(y, |y|/2)$ and $D=D_1\cup D_3$.
Let $F= B(y/2, |y|/2)$. Observe that for $ u\in D_1 \cap F$ and $ z\in D_3 \cap F$ we have $|z-u| \le |y|$, hence by geometric considerations,
\begin{align*}
  f^0(s,z)&=\int_{{D_1}}p_{{D_1}}(s, 0, u)\nu(z-u)du\ge \nu(y) \int_{{D_1}\cap F}p_{{D_1}}(s, 0, u)du\\
 &\ge 2^{-d}\nu(y) \int_{{D_1}}p_{{D_1}}(s, 0, u)du=2^{-d} \nu(y) \p^0(\tau_{D_1}>s),
\end{align*}
and
\begin{align*}
  p(t,0, y)&\ge\int_0^t\int_{D_3\cap F}f^0(s,z)p(t-s, z, y)dz ds
\ge
  2^{-d} \nu(y)\p^0(\tau_{D_1}>t)
  \int_0^t\int_{D_3\cap F}p(t-s, z, y)dz ds\\
  &
\ge
  4^{-d} \nu(y)\p^0(\tau_{D_1}>t)
  \int_0^t\int_{D_3}p_{D_3}(t-s, z, y)dz ds\\
&= 4^{-d}
  \nu(y)\ \p^0(\tau_{D_1}>t)
  \int_0^t\p^y(\tau_{D_3}>s)ds
\ge 4^{-d} t\,\nu(y) \left[\p^0(\tau_{D_1}>t)\right]^2.
\end{align*}
\end{proof}

We shall study in detail the factors in the
inequalities
of
Lemma~\ref{lemppu100}.
\begin{lem}\label{kula1}
$\Ci=\Ci(d)$, $\Cj=\Cj(d)$ and $\Ck =\Ck (d)$ exist such that for all $t,r> 0$ and $|x|\le r/2$,
\begin{align}\label{eq:l}\p^x\left(|X_{\tau_{D}}|\ge r\right)
&\le \Ci \frac{ \E^x\tau_{D}}{V^2(r)},\\
\p^x(\tau_{B_r}\le t  )&\le \Cj\frac t{V^2(r)},\label{eq:2}\\
\p^x(\tau_{B_r}>\Ck V^2(r) )&\ge 1/2.\label{eq:3}
\end{align}
\end{lem}
\begin{proof}
The result combines Lemma~2.7 and Corollary~2.8 of  \cite{BGR2}.
\end{proof}
\begin{cor}
There is $c=c(d)>0$ such that if $y\in \Rdz$ and $0<t<c/\psi(1/|y|)$, then $p_t(y)\geq 4^{-d-1}t\nu(y)$.
\end{cor}
\begin{proof}
The result follows from
Lemma~\ref{HKLB} and \eqref{eq:3}, by Lemma~\ref{ch1V} and subadditivity of $V$.
\end{proof}

Following \cite{BGR2}, for $r>0$ we define
\begin{equation}\label{defIJ}
\mathcal{I}(r)=\inf_{ 0<\rho\leq r/2}\nu(B_{r}\setminus B_\rho) V^2(\rho)\qquad \text{and}\quad  \mathcal{J}(r)=\inf_{ 0<\rho\leq r}
\nu(B_\rho^c)V^2(\rho).
\end{equation}
The quantities are meant to simplify notation in arguments leading from Ikeda-Watanabe formulas to estimates of the survival probability from below, where $\mathcal{I}$ is used, and to estimates of the expected exit time from above, where $\mathcal{J}$ is used. Note that by Lemma~\ref{ch1V}, $h(r)V^2(r)\approx 1$. Below we strive for lower bounds for $\mathcal{I}$ and $\mathcal{J}$. Such bounds can be interpreted as comparability of a part of the integral defining $h$ with the whole, cf. \eqref{def:GKh2}, and certainly, $\mathcal{I}$ and $\mathcal{J}$ describe
the size of
the L\'evy measure in comparison to $V^{-2}$ and $h$.
Additional information on $\mathcal{I}$ is given in Lemma~\ref{Levy_V} below.

The following result is taken from \cite[Proposition 6.1]{BGR2}.
\begin{lem}\label{L5a1}
Let  $\A$ hold. There are $ \Cb=\Cb (d)<1$ and $\Cc=\Cc (d)$ such that for $r>0$,
\begin{eqnarray*}\p^x(\tau_{B_{r}}>t)&\ge&   \Cc\,\frac{\mathcal{I}(r)}{H_r}\left(\frac{V(\delta_{B_{r}}(x))}{\sqrt{t}}\wedge1\right), \qquad
0<t\le \Cb V^2(r), \quad x\in \Rd.
\end{eqnarray*}
\end{lem}
In the next result we slightly extend \cite[Remark~8]{BGR2}, to include processes with {\it local} scalings.
\begin{lem}\label{lem:spgenC11}
If $D$ is  $C^{1,1}$ at scale $r$, $\nu(r)>0$,
and $\psi\in \mathrm{WLSC}\cap\mathrm{WUSC}$, then
$$\p^x(\tau_D>t)\approx \frac{V(\delta_D(x))}{\sqrt{t}}\wedge 1, \qquad  0<t\leq C_5V^2(r), \quad x\in \Rd.$$
The comparison depends only on $X$ and $r$. If
the scalings are global, then the comparison depends only on $d$ and the scaling characteristics of $\psi$.
\end{lem}
\begin{proof}
Let $x\in D$.
If $\delta_D(x)\geq r/2$, then there is a ball $B\subset D$ with radius $r$ such that $\delta_B(x)\geq r/2$. By Lemma \ref{L5a1} and subadditivity of $V$ we obtain $$\p^x(\tau_D>t)\geq \p^x(\tau_{B}>t)\geq c_2 \(\frac{V(r/2)}{\sqrt{C_5}V(r)}\wedge 1\)> c_2/2,$$
thus
$\p^x(\tau_D>t) \approx 1 \approx \frac{V(\delta_D(x))}{\sqrt{t}}\wedge 1$.
Here $c_2>0$ depends only on $X$ and $r$ (on $d$ and $\psi$ if global scalings hold), as
follows from \cite[Proposition 5.2(ii) and  Lemma 7.3]{BGR2}.
Namely, \cite[Proposition 5.2(ii)]{BGR2} yields $\inf_{0<s\le R} \mathcal{I}(s)>0$ for some $0<R\le r$.
Since $\mathcal{I}(r)\ge \mathcal{I}(R)\wedge [\nu(r)|B_r\setminus B_{r/2}|V^2(R/2)]$, we obtain $\mathcal{I}(r)>0$. On the other hand, \cite[Lemma 7.3]{BGR2} yields $\A$ hence $H_r<\infty$, as needed.
In the case of global scalings, \cite[Lemma 7.3 and the proof of Proposition 5.2(ii)]{BGR2} show that
$c_2$ only depends on $d$ and the parameters $\la$, $\lC$, $\ua$ and $\uC$ of the scalings (by
\cite[Theorem~26]{2014-KB-TG-MR-jfa} we then automatically have $\nu(r)>0$).
If local scalings are only assumed, then \cite[Proposition 5.2 and  Lemma 7.3]{BGR2}
yield $c_2=c_2(d,\psi)$ if $r$ is small (the notion of  {\it smallness} depends on the characteristics of the scalings).

If $\delta_D(x)<r/2$ and $S\in\partial D$ is such that $\delta_D(x)=|x-S|$, then there are balls $B$ and $B'$ with radii $r$, tangent at $S$ and such that $B\subset D\subset \overline{B'}^c$. Since $\delta_D(x)=\delta_B(x)=\delta_{\overline{B'}^c}(x)$, by  Lemma \ref{L5a1} and \cite[Lemma 6.2]{BGR2} we get the claim, see also \cite[Proposition 5.2 and  Lemma 7.3]{BGR2}.
\end{proof}
The above lemmas largely resolve the asymptotics of the survival probability in $C^{1,1}$ open sets in small time. Estimates of the survival probability for large time depend on specific geometry of $D$ at infinity and shall be studied later on in this paper.

The following result relates survival probabilities to the scenario of $X$ evading the complement of $D$ by going towards the center of the set.
\begin{lem}\label{lower bound13} Let $0<r\le 1$,
 $x\in B_1$ and $\delta_{B_1}(x)<r/6$. Denote $x_0=x/|x|$, $x_1 = x_0(1-r/2)$
 and  $F_x= B(x_0,r/4)\cap B_1
$.
  There is a constant $c=c(d)$
such that
  \begin{eqnarray}\label{eq:iobc} \int_{B(x_1,r/12)}p_{B_1}(t, x, v)
dv\ge c \, t\,\nu(r)r^d \p^x(\tau_{F_x}>t)\p^0(\tau_{B_{r/12}}>t), \qquad t>0.
  \end{eqnarray}
\end{lem}
\begin{proof}  We use  Lemma \ref{lemppu100}
  with  $D=B_1$, $D_1= F_x$, $D_3=B(x_1,r/6)$. For $v\in B(x_1, r/12)$,
  \begin{eqnarray*}p_{B_1}(t, x, v)&\ge&
    t\,\p^x(\tau_{D_1}>t)\p^v(\tau_{D_3}>t)\!\inf_{w\in D_1,\, z\in D_3}\!\!\!\!\nu(z-w) \\
		 &\ge& t\,\nu(r) \p^x(\tau_{D_1}>t)\p^0(\tau_{B_{r/12}}>t).
  \end{eqnarray*}
Integrating against $v\in B(x_1, r/12)$ we obtain \eqref{eq:iobc} with $c=\omega_d(12)^{-d}/ d$.
  \end{proof}

\begin{cor}\label{int2}
Assume that $\A$ holds, $0<r\le 1$  and $x\in B_1$.
Let $x_1 = x$, if $\delta_{B_1}(x)\ge r/6$, otherwise let $x_1 = x(1-r/2)/|x|$.
    There are  constants $\Cd=\Cd (d)$ and $\Ce=\Ce (d)$  such that if  $0<t\le  \Cd V^2(r)$, then
 $$ \int_{B(x_1, r/12)}p_{B_1}(t, x, v)dv \ge \Ce \frac{\mathcal{I}({r/8})}{H_{1}}  \,t\,\nu(r)r^d \left(\frac{V(\delta_{B_1}(x))}{\sqrt{t}}\wedge1\right).$$
\end{cor}
\begin{proof}
Let $\Cd=\Ck /(12)^2\wedge \Cb/8^2$ and $0<t\leq \Cd V^2(r)$. If $0<\delta_{{B_1}}
(x)<r/6$, then by Lemma~\ref{lower bound13} and  Lemma~\ref{kula1},
$$ \int_{B(x_1, r/12)}p_{B_1}(t, x, v)dv \ge c\, t\, \nu(r)r^d\p^x(\tau_{B(x(1-r/8){/|x|},r/8)}>t).$$
By Lemma \ref{L5a1} we get the result, since $V(\delta_{B_1}(x)\wedge r/8)\ge
V(\delta_{B_1}(x)\wedge r/4)/2=V(\delta_{B_1}(x))/2$.
If $\delta_{{B_1}}(x)\geq r/6$, then by \eqref{eq:3},
$$ \int_{B(x_1, r/12)}p_{B_1}(t, x, v)dv\geq  \int_{B(x, r/12)}p_{B(x,r/12)}(t, x, v)dv=\p^0(\tau_{B_{r/12}}>t)\geq 1/2.$$
By \cite[(16)]{2014-KB-TG-MR-jfa} and \eqref{cVh1pg}, $t\nu(r)r^d\leq c(d)$ for $t\leq \Cd V^2(r)$. This ends the proof.
\end{proof}

\begin{lem}\label{exit_time_R}
Assume that
$\A$ holds.  Let $R>0$ and
$D=\overline{B}_{R}^c$.  Let $0<r<R$, $x\in D$, $0<\delta_D(x)\le r/2$,  $x_0=xR/|x|$ and
$D_1= B(x_0, r)\cap D$. Then
\begin{equation}\label{seas1}
\E^x \tau_{D_1}\le \Ca\frac{H_R}{\mathcal{J}(R)^{2}}  V(\delta_D(x))\,V(r).
\end{equation}
Furthermore,
$\Ca H_R/\mathcal{J}(R)^{2}\geq 1/(2\Ci)$ and $\Ci\geq 1/2$.
\end{lem}
\begin{proof}\eqref{seas1} was proved in \cite{BGR2}, see Corollary 4.5 ibid., but
we need to justify the statement about the constants.
  Let $|x|=5R/4$ and $D_1=B(4x/5,R/2)\cap (B_R)^c$. Since $B(x,R/4)\subset D_1$, by \eqref{eq:l} and \cite[Corollary 4.1]{BGR2} we obtain
$$\Ci^{-1}V^2(R/4)\leq \E^x\tau_{D_1}\leq \Ca H_R\mathcal{J}(R)^{-2} V(R/4)V(R/2).$$
This and subadditivity of $V$ imply  $$\Ci\Ca H_R\mathcal{J}(R)^{-2}\geq \frac{V(R/4)}{V(R/2)}\geq \frac{1}{2}.$$ Due to \eqref{eq:l} and \cite[Lemma 2.3]{BGR2}, $\Ci\geq 1/2$.
\end{proof}

\section{Upper bound }\label{sec:UpperBound}
In this section we shall study consequences of the following structure
assumption:
\begin{align}\label{eq:se}
p_t( x)\le t  F(|x|), \qquad t>0,\ x\in \Rdz,
\end{align}
where $F$  is a nonnegative nonincreasing function on $(0,\infty)$.
We shall use \eqref{eq:se} to estimate the heat kernel $p_D(t,x,y)$ of $C^{1,1}$ sets $D\subset \Rd$ for $t>0$ and $x,y\in \Rd$. In view of
Lemma~\ref{upper_den}, we
may think of $F(r)=\Cf/[r^dV^2(r)]$ here (the method however seems to generalize beyond the context of the present paper).
We note that $p_D(t,x,y)=0$ if $x\in D^c$ or $y\in D^c$, cf. Remark~\ref{rem:pDz}, so without much mention in what follows we only consider $x,y\in D$ and $D\neq \emptyset$.
We start with
the following upper bound, which elaborates Lemma~\ref{lemppu100}  for
the complement of the ball.
\begin{thm}\label{heatKernelCompl}
Let  $\A$ hold, $R>0$ and $D=\overline{B}^c_R$.
There is
$C=C(d)$
such  that  if
\eqref{eq:se} is true with nonincreasing function
$F\ge 0$ on $(0,\infty)$, $0<t\le V^2(|x-y|)$ and $x,y\in D$, then
\begin{align}\label{eq:ubcb} p_{D}(t, x, y)
    \le  C \frac{H_R^2}{\mathcal{J}(R)^{4}}  \left(\frac{V(\delta_D(x))}{\sqrt{t}\wedge V(R)}\wedge1 \right) \left(\frac{V(\delta_D(x))}{\sqrt{t}\wedge V(R)}\wedge1 \right)tF(|x-y|/9).
\end{align}
  \end{thm}

\begin{proof}

Let {$t_0= t\wedge V^2(R)$} and $x, y\in D=\overline{B}_R^c$. We choose $r>0$ so that $V({12}r)= \sqrt{{t_0}}$. In particular, $r\le R/12$.
If $\delta_D(x)\wedge\delta_D(y)\ge r/3$, then \eqref{eq:ubcb}
is
verified as follows.
Since $\delta_D(x)\geq r/3$, by subadditivity of $V$ we have ${V(\delta_D(x))}/{\sqrt{t_0}}\ge 1/{36}$.  Thus,
\begin{equation}\label{HKC1}
p_D(t,x,y)\leq 36\(\frac{V(\delta_D(x))}{\sqrt{t_0}}\wedge 1\)p(t,x,y)  \leq 36\(\frac{V(\delta_D(x))}{\sqrt{t_0}}\wedge 1\)tF(|x-y|/3).
\end{equation}
By Lemma \ref{exit_time_R} we have ${H_R^2}/{\mathcal{J}(R)^{4}}\ge 1/{(4\Ci^2 \Ca^2)}$. Hence for $\delta_D(x)\wedge\delta_D(y)\ge r/3$ we have
 $$ p_{D}(t, x, y)
    \le  4\cdot 36^2 \Ci^2 \Ca^2{\frac{H_R^2}{\mathcal{J}(R)^{4}}}  \left(\frac{V(\delta_D(x))}{\sqrt{t_0}}\wedge1 \right) \left(\frac{V(\delta_D(x))}{\sqrt{t_0}}\wedge1 \right){tF(|x-y|/9)}.
$$

We now may and do
assume that  $0<\delta_D(x)< r/3$,
hence ${V(\delta_D(x))}/{\sqrt{t_0}}< 1$.  At first, we also assume that $ V^2({3}|x-y|)\ge t$, in particular   $|x-y|\ge {4}r$. We define
$$x_0=Rx/|x|,\quad D_1= B\(x_0, r\)\cap D,\quad  D_3= B(x, {2}|x-y|/{3})^c\cap D. $$
Note that $|z-y|\ge |x- y|/{3}$ if $z\in D_2= D\setminus(D_1\cup D_3)$.     Radial monotonicity  of $p_t$  implies
$$\sup_{s<t,\, z\in D_2} p(s, z, y)\le    tF(|x-y|/3). $$
If  $u\in D_1,\, z\in D_3$, then $|z-u|\ge {2}|x- y|/{3 -| x-x_0|-|x_0-u|}{>} |x- y|/{3}$. Hence,
$$\sup_{u\in D_1,\, z\in D_3}\nu(z-u)\le  \nu((x- y)/{3}){\leq F(|x-y|/3)}.$$
By Lemma \ref{lemppu100},
\begin{eqnarray*}
    p_{D}(t, x, y)
    &\le& \({t}\, \p^x(X_{\tau_{D_1}}\in D_2)+ \E^x \tau_{D_1}\){F(|x- y|/3)}
    .
  \end{eqnarray*}
   By (\ref{eq:l}) and {subadditivity of $V$},
  \begin{equation}\label{exit1}\p^x(X_{\tau_{D_1}}\in D_2)\le \Ci\frac{   \E^x \tau_{D_1}}{V^2(r)}\le {144}\Ci\frac{   \E^x \tau_{D_1}}{{t_0}}. \end{equation}
		 By {Lemma \ref{exit_time_R}},
  \begin{equation}\label{exit2}\E^x \tau_{D_1}\le {\Ca\frac{H_R}{\mathcal{J}(R)^{2}}} V(r)V(\delta_D(x))\le  {\Ca\frac{H_R}{\mathcal{J}(R)^{2}}\sqrt{t_0}}V(\delta_D(x)).\end{equation}
  We let $c_1=(144\Ci+1)\Ca$ and obtain
     \begin{equation}\label{HKC2}
    p_{D}(t, x, y)
    \le { c_1 \frac{H_R}{\mathcal{J}(R)^{2}}}\frac{V(\delta_D(x))}{{\sqrt{t_0}}}
  {tF(|x- y|/3)}.
  \end{equation}

Combining \eqref{HKC1}, \eqref{HKC2} and Lemma \ref{exit_time_R} we see that
 \begin{equation*}
    p_{D}(t, u, v)
    \le { c_2 \frac{H_R}{\mathcal{J}(R)^{2}}}\(\frac{V(\delta_D(u))}{{\sqrt{t_0}}}\wedge 1\)
    tF(|u- v|/3),
  \end{equation*}
where $u,v\in D$, $V^2(3|u-v|)\geq t$ and $c_2= c_1\vee (72 \Ci \Ca)$.
By symmetry,
 $$ p_{D}(t, u, v)
    \le C^* \left(\frac{V(\delta_D(v))}{\sqrt{t_0}}\wedge1 \right){tF(|u- v|/3)},$$
		where $C^*={ c_2 H_R\mathcal{J}(R)^{-2}}$.
		
		We observe that ${s}\le V^2({3}|z-y|)$ if ${s\leq}t\le V^2(|x-y|)$ and $z\in D_2$. By  previous estimate,
				 $$\sup_{s<t,\,z\in D_2} p_{D}(s, z, y)\le
    C^* \left(\frac{V(\delta_D(y))}{\sqrt{t_0}}\wedge1 \right)tF(|x- y|/{9}).$$
		Applying  Lemma \ref{lemppu100} and the estimate $\sup_{u\in D_1,\, z\in D_3}\nu(z-u)\le \nu((x-y)/3)$, we obtain
		\begin{eqnarray*}
    p_{D}(t, x, y)
    &\le& \p^x(X_{\tau_{D_1}}\in D_2)\sup_{s<t,\, z\in D_2} p_D(s, z, y)\\
    &+&  \left((t\wedge \E^x \tau_{D_1})\p^y(\tau_{D}>t/2)+ \p^x(\tau_{D_1}>t/2)\int_0^{t/2}\p^y(\tau_{D}>s)ds\right)\nu((x-y)/3)\\
		&=& I_1 +I_2 .
  \end{eqnarray*}
	Combining (\ref{exit1}) and (\ref{exit2}) we prove $$\p^x(X_{\tau_{D_1}}\in D_2)\le C^*  \left(\frac{V(\delta_D(x))}{\sqrt{t_0}}\wedge1 \right).$$
	Therefore,
	$$I_1\le  ( C^*)^2 \left(\frac{V(\delta_D(y))}{\sqrt{t_0}}\wedge1 \right) \left(\frac{V(\delta_D(x))}{\sqrt{t_0}}\wedge1 \right)tF(|x-y|/9).$$
	By (\ref{exit2}) we obtain
	$$ t\wedge \E^x \tau_{D_1}\le \frac{C^*}{{73}} \left(\frac{V(\delta_D(x))}{\sqrt{t_0}}\wedge1 \right)t.$$
	From \cite[Lemma 6.2 and its proof]{BGR2} and  {Lemma \ref{exit_time_R}} it is clear that
	$$\int_0^{t/2}\p^y(\tau_{D}>s)ds\le \frac{{6}}{{146}}C^* \left(\frac{V(\delta_D(y))}{\sqrt{t_0}}\wedge1 \right)t$$
	and
	$$\p^x(\tau_{D_1}>t/2)\le\p^y(\tau_{D}>t/2)\le {\frac{3}{73}}C^* \left(\frac{V(\delta_D(y))}{\sqrt{t_0}}\wedge1 \right).
$$
	The  estimates  imply that	
	 $$I_2\le ( C^*)^2 \left(\frac{V(\delta_D(y))}{\sqrt{t_0}}\wedge1 \right) \left(\frac{V(\delta_D(x))}{\sqrt{t_0}}\wedge1 \right)t\nu((x-y)/3).$$
	Finally, for  $t\le V^2(|x-y|)$ we have
	$$	p_{D}(t, x, y)\le {2}( C^*)^2 \left(\frac{V(\delta_D(y))}{\sqrt{t_0}}\wedge1 \right) \left(\frac{V(\delta_D(x))}{\sqrt{t_0}}\wedge1 \right)tF(|x-y|/9).$$

		\end{proof}

	\begin{rem} \label{heatKernelCompl_2balls}
With cosmetic adjustments, the proof also works for $D=\left(\overline{B(Q_1,R)}\cup \overline{B(Q_2,R)}\right)^c$, where $Q_1,Q_2\in \Rd$. Then by domain monotonicity of heat kernels,
the conclusion of Theorem~\ref{heatKernelCompl}  holds for every open
set $D$
having the outer ball property at scale $R$.
  \end{rem}

Here is an analogue of
Theorem \ref{heatKernelCompl}
for
convex sets.
Noteworthy, we do not assume $\A$ (or scalings) here.
		\begin{thm}\label{HeatKula} Suppose that  $ p(t, x)\le  t  F(|x|), \ t>0, \ x\neq 0$, with nonincreasing $F\ge 0$.
Let $D$ be open and convex.
There  is {$C=C(d)$}
such that if $x,y\in D $ and $0<t\le V^2(|x-y|)$, then
		$$p_{D}(t, x, y)
    \le  C  \left(\frac{V(\delta_D(x))}{\sqrt{t} }\wedge1 \right) \left(\frac{V(\delta_D(y))}{\sqrt{t}}\wedge1 \right) \left( p_{t/2}(0) \wedge tF(|x-y|/9)\right).$$
		\end{thm}
		
		\begin{proof}
By {convexity of $D$} and \cite[(2.21)]{BGR2} there is an absolute constant $c$ such that
 \begin{equation}\label{kula-exit} \p^x(\tau_D>t)\le  c \left(\frac{V(\delta_D(x))}{\sqrt{t} }\wedge1 \right). \end{equation}
 Hence,  by Lemma \ref{UB} and subadditivity of $V$, $$p_{D}(t, x, y)
    \le  4c^2  \left(\frac{V(\delta_D(x))}{\sqrt{t} }\wedge1 \right) \left(\frac{V(\delta_D(y))}{\sqrt{t}}\wedge1 \right) p_{t/2}(0).$$
    This  provides the first part of the conclusion.
The full conclusion follows  by the proof of  Theorem \ref{heatKernelCompl}
with some modifications. We  {fix $x_0$ such that $\delta_D(x)=|x-x_0|$ and } define $D_1, D_2$ and $D_3$ exactly in the same way as in the proof of  Theorem \ref{heatKernelCompl}. To validate all the arguments we need appropriate  estimates of $\p^x(\tau_D>t)$ and $\E^x\tau_{D_1}$.  Note that (\ref{kula-exit}) provides a desired estimate for  $\p^x(\tau_D>t)$, while $\E^x\tau_{D_1}\leq \E^{\delta_D(x)}\tau^{X^{(1)}}_{(0,r)}\leq V(\delta_D(x))V(r)$, where $\tau^{X^{(1)}}$ is the first exit time of the first coordinate of $X$ \cite[Proposition 3.5]{2012GR}.
With these estimates at hand,
we may replace the  constant $H_R\mathcal{J}(R)^{-2}$ used in the proof of Theorem \ref{heatKernelCompl} (see, e.g., \eqref{HKC2})  by a constant depending only on $d$.
\end{proof}

As we already indicated, the above two theorems
apply to every pure-jump  unimodal  L\'evy process with infinite L\'evy measure:
by Lemma~\ref{upper_den},
we can take $F(r)= \Cf/[r^dV^2(r)]$, to obtain the following consequences of Theorem \ref{heatKernelCompl}.

	\begin{cor}\label{heatKernelCompl1}

Let $D$ be an open set satisfying the outer ball condition at a scale $R$.
There is a constant $C=C(d)$ such that   for all $x,y\in D$ and  $t\le V^2(|x-y|)$,
 $$ p_{D}(t, x, y)
    \le  C  \frac{H_R^2}{\mathcal{J}(R)^{4}}  \left(\frac{V(\delta_D(x))}{\sqrt{t}\wedge V(R)}\wedge1 \right) \left(\frac{V(\delta_D(x))}{\sqrt{t}\wedge V(R)}\wedge1 \right)\left(p_{t/2}(0) \wedge \frac {t}{V^2(|x- y|) |x- y|^d}\right),$$
provided $\A$ holds.
If, additionally,
$\psi\in\WLSC{\la}{ \lt}{\lC}$, then for all $t>0$,
$$ p_{D}(t, x, y)
    \le  \frac{C}{\lC^{1+d/\la}}    \frac{H_R^2}{\mathcal{J}(R)^{4}}\left(\frac{V(\delta_D(x))}{\sqrt{t}\wedge V(R)}\wedge1 \right) \left(\frac{V(\delta_D(x))}{\sqrt{t}\wedge V(R)}\wedge1 \right)\left(p_{t/2}(0) \wedge \frac {t}{V^2(|x- y|) |x- y|^d}\right),$$
provided $|x-y|< 1/\lt$.
Here $C=C(d, \la)$.
\end{cor}
\begin{proof} Let $x\in D$ and $S\in \partial D$ such that $\delta_D(x)=|x-S|$.
Since $D$ satisfies the outer ball condition, there is a ball $B$ of radius $R$ such that  $B\subset D^c$  and $S\in \overline{B}$ and $\delta_D(x)= \delta_{\overline{B}^c}(x)$.
 By  \cite[Lemma 6.2]{BGR2},
$$\p^x\(\tau_{D}>\frac{t}{4}\)\le \p^x\(\tau_{\overline{B}^c}>\frac{t}{4}\)\le c \frac{H_R}{\mathcal{J}(R)^{2}} \left(\frac{V(\delta_D(x))}{\sqrt{t}\wedge V(R)}\wedge1 \right),$$
with $c=c(d)$. Hence,  the first bound in the statement
is a simple consequence of  Lemma \ref{UB},
 Theorem  \ref{heatKernelCompl} and Remark  \ref{heatKernelCompl_2balls}.
To prove the second one
  we only need to consider the case $t\ge V^2(|x-y|)$.
  Let $t_0=  V^2(|x-y|)$. Since  $|x-y|\le 1/ {\lt}$ we have $t_0\le   V^2(1/ {\lt})$. Applying Lemma \ref{sup_p_t} we obtain
 $$ p_{t/2}(0)\le p_{t_0/2}(0) \le C\frac1 {\lC^{1+d/\la}} \[V^{-1}(\sqrt{t_0})\]^{-d}= C\frac1 {\lC^{1+d/\la}}\frac 1{|x-y|^d} \le C\frac1 {\lC^{1+d/\la}}\frac t{V^2(|x-y|)|x-y|^d},$$ with $C=C(d,\la)$.
This ends the proof
due to Lemma \ref{UB}.
\end{proof}
Regarding the assumptions of Corollary~\ref{heatKernelCompl1} we recall that $\A$ holds automatically if $\psi\in$ WLSC and $d\ge 3$.
\begin{rem}\label{ScAgree}
In what follows, when we write $\psi\in \mathrm{WLSC}\cap\mathrm{WUSC}$ and $C=C(\psi,\ldots)$, we mean $\psi\in\WLSC{\la}{\lt}{\lC}\cap \WUSC{\ua}{\ut}{\uC}$ and $C=C(\la,\lt,\lC,\ua,\ut,\uC,\ldots)$. Here is a simplifying convention 
\end{rem}
\begin{thm}\label{heatKernelComplGlobal}Let $R>0$ and let $D$ be an open set satisfying the outer ball condition at scale $R$.
Suppose that  global  {\rm WLSC} and {\rm WUSC} hold for $\psi$. Then there is a constant $C=C(d,\psi)$ such that for all $t>0$ and  $x,y\in D$,
$$ p_{D}(t, x, y)
    \le C  \left(\frac{V(\delta_D(x))}{\sqrt{t}\wedge V(R)}\wedge1 \right) \left(\frac{V(\delta_D(y))}{\sqrt{t}\wedge V(R)}\wedge1 \right)p(t,x,y)
.$$
  \end{thm}
  \begin{proof}Due to \cite[Corollary 24]{2014-KB-TG-MR-jfa}
$p_t(0)\approx
p_{t/2}(0)$ and   by Lemma \ref{sup_p_t} we have $p_t(0)\approx  \[V^{-1}(\sqrt{t})\]^{-d}$ with  comparability constants depending only on $d$ and $\psi$. By  Lemma~\ref{densityApprox}, $p(t,x,y)\approx  p_{t/2}(0) \wedge \left[t |x- y|^{-d}/V^2(|x- y|)\right]$ with comparability constants depending only on $d$ and $\psi$.

 By global WLSC and WUSC for $\psi$, we have
  $ \inf_{R>0} \mathcal{J}(R)>0$ {(see \cite[Proposition 5.2]{BGR2}) and $\As$  (see \cite[Lemma~7.2 and Lemma 7.3]{BGR2})}, hence
  ${H_R}/{\mathcal{J}(R)^{2}}\le C=C(d,\psi)$.

Therefore 	the claim is an obvious consequence of the second bound of Corollary \ref{heatKernelCompl1}.
\end{proof}
\begin{rem}\label{rem:d1s}
$\A$ may usually be circumvented in the (exceptional) dimension $d=1$, cf. \cite[Proposition~2.6, Corollary~4.7]{BGR2}. This may be of interest for the upper bounds of the survival probability if $\psi$ satisfies WLSC but not WUSC.
\end{rem}

The following is a simple corollary to Theorem \ref{HeatKula}. We skip the proof, since it repeats the arguments used to prove Corollary \ref{heatKernelCompl1}
	
		\begin{cor}\label{hk_kula1}
Let
$D$ be  open and convex.
			 If $\psi\in\WLSC{\la}{\lt}{ \lC}$, $t>0$, $|x-y|<  1/\lt$ and $x,y\in D$, then    there  is a constant $C=C(d,\la)$ such that
			$$p_{D}(t, x, y)
    \le  \frac{C} {\lC^{2(1+d)/\la+1}}  \left(\frac{V(\delta_D(x))}{\sqrt{t} }\wedge1 \right) \left(\frac{V(\delta_D(y))}{\sqrt{t}}\wedge1 \right) \left(p_{t/2}(0) \wedge \frac {t}{V^2(|x- y|) |x- y|^d}\right),$$
and if $\lt=0$, then			
there is a constant $C=C(d,\psi)$ such that  for all $t>0$ and  $x,y\in D$,
		$$p_{D}(t, x, y)
    \le  C  \left(\frac{V(\delta_D(x))}{\sqrt{t} }\wedge1 \right) \left(\frac{V(\delta_D(y))}{\sqrt{t}}\wedge1 \right) \left(p_t(0) \wedge \frac {t}{V^2(|x- y|) |x- y|^d}\right).$$
		\end{cor}

\section{Lower bound}\label{sec:LowerBound}
By Lemma~\ref{upper_den},
$p_t(x)\le {\Cf}t/[V^2(|x|)|x|^d].$
We shall often assume the following partial converse.

\noindent
{\bf Condition ${\bf G}_R$:}
We say ${\bf G}_R$ holds if $R>0$ and there is $\Cl \in [1,\infty)$ such that
\begin{equation}\label{A}\frac{t}{ V^2(|x|)|x|^d}\le \Cl p_t(x),
\qquad
{0<}t\le V^2(|x|),\quad |x|\le R.
\end{equation}
The condition is merely for notational convenience since it has the following characterization.
\begin{lem}\label{Gr_Ex} Let $0<R<\infty$.
${\bf G}_R$ holds if and only if $\psi\in\mathrm{WLSC}\cap\mathrm{WUSC}$ and  $\nu(R^-)>0$.
\end{lem}
\begin{proof}
For one implication we assume that $\psi\in\mathrm{WLSC}\cap\mathrm{WUSC}$ and  $\nu(R^-)>0$.
By Lemma \ref{densityApprox} there is $r=r(d,\psi)>0$ such that ${\bf G}_{r}$ holds.
We may and do assume that $R>r$.
Let $r\leq |x|\le R$.  By Lemma~\ref{HKLB}, and continuity of $p_t$, for $0<t\leq V^2(R)$ we have
$$p_t(R)\geq 4^{-d} t\nu(R^-)[\p^0(\tau_{B_{R/4}}>V^2(R))]^2= c t .$$
By radial monotonicity of $p_t$,
$$p_t(x)\geq p_t(R)\geq   c  V^2(r)r^d\frac{t}{V^2(|x|)|x|^d},$$
as needed.
For the converse implication, we note that ${\bf G}_R$ and \cite[Theorem~26]{2014-KB-TG-MR-jfa} imply scalings of $\psi$ with $\theta=R^{-1}$. Since $p_t(x)/t\to \nu(x)$ vaguely on $\Rdz$, ${\bf G}_R$ and monotonicity of $\nu$ yield $\nu(R^-)\ge \Cl^{-1}/[V^2(R)R^d]>0$,
which ends the proof.
\end{proof}
Thus in many cases, if ${\bf G}_R$ holds for some  value $R$ and $\Cl$, then it holds for every $R{\in (0,\infty)}$ with $\Cl $ depending on $R$. This is so, e.g., for every
subordinate Brownian motion,
due to
Lemma \ref{Gr_Ex} and positivity of $\nu$.
It may also 
happen that \eqref{A} is true for some $R$, but it fails for larger values of $R$.
This is the case for the truncated L\'evy process, whose
L\'evy measure is supported by a
bounded set (see Section~\ref{sec:examples}).
For clarity, ${\bf G}_\infty$ is equivalent to global scaling conditions on $\psi$ \cite[Theorem~26]{2014-KB-TG-MR-jfa}.
 Notice also that due to  \cite[Theorem 26]{2014-KB-TG-MR-jfa} and \cite[Lemmas 7.2 and 7.3]{BGR2},
${\bf G}_R$
implies $\A$.
Furthermore, if we replace $X$ by $X/R$, then by \eqref{eq:tLV}, $V(x)$ is replaced by $V(Rx)$, and if we subsequently replace $x$ by $Rx$,
then we equivalently obtain ${\bf G}_1$
for $X/R$.

Before stating the next result we recall that ${\mathcal I}$ is defined in \eqref{defIJ}.
\begin{lem} \label{Levy_V} There is $c=c(d)$ such that if ${\bf G}_R$
holds,  then
$$\inf_{r\le R}\mathcal{I}(r)\ge c/\Cl .$$
\end{lem}
 \begin{proof} Let $0<r\le R$. Note that   (\ref{A}) implies  \begin{equation}\label{Levy-V1}\nu(x)\ge \frac {1}{\Cl  V^2(|x|)|x|^d}, \quad |x|<R.\end{equation}
For $\rho\le r/2$ we obtain
$$V^2({\rho}) \nu(B_{{r}}\setminus B_{\rho})\ge V^2(\rho)\int_{B_{2\rho}\setminus B_{\rho}} \frac {dx}{\Cl  V^2(|x|)|x|^d}\ge \frac14\int_{B_{2\rho}\setminus B_{\rho}} \frac {dx}{\Cl  |x|^d} = \frac c { \Cl }, $$
which completes the proof.
\end{proof}

We now  give the lower bound for the heat kernel
for union of two balls of the same radius.
\begin{thm} \label{lower_hk_estimate}
Let $R>0$ and
$\psi\in \WLSC{\la}{R^{-1}}{\lC}$. Assume that   ${\bf G}_R$ is satisfied.  Let $D=B(z_1,R)\cup B(z_2,R)$.
 There exist $c=c(d)<1, c_1=c_1(d,\la)$
 such  that
		$$ p_{D}(t,x,y)\ge   \frac{ c_1 \lC^{1+d/\la}}  {{H_R^2}(\Cl )^{{{9}+d/\la}}}  \left(\frac{V(\delta_D(x))}{\sqrt{t}}\wedge1\right)\left(\frac{V(\delta_D(y))}{\sqrt{t}}\wedge1\right)
		\big( p_{t/2}(0)\wedge [t\,\nu(2|x-y|\wedge \diam(D))]\big),$$
provided  $0<t\le  c V^2(R)/\Cl $,
$x,y\in D$, $\delta_D(x)= \delta_{B(z_1,R)}(x)$ and $\delta_D(y)= \delta_{B(z_2,R)}(
y)$.
\end{thm}
 \begin{proof}By the discussion at the beginning of the section we may and do assume that $R=1$.
We may {also} assume that $z_1=0$. Let
\begin{equation}\label{eq:dcg}
 c^*= \frac{\min\{( {12^{d+4}3^{d+3}2\Cj\Cf} )^{-1},\, \Cb/{36},\, \Cd \}} {\Cl },
\end{equation}
where  $\Cb<1, \Cd$ are
from  Lemma \ref{L5a1} and Corollary \ref{int2} and $\Cl $ is from \eqref{A}.
Let
$0<t\le (c^*/9) V^2(1)
\le c^*V^2(1/3)$.  Let $0<r\le 1/3$ be such that
$$t=c^*V^2(r),\quad {\text{ or }\quad r=V^{-1}(\sqrt{t/c^*})}.$$
Let $x\in D$.   If  $\delta_D(x)<r/6$, then we let $x_0=x/|x|$, $x_1 = x_0(1-r/2)$ {and $r_x=r/2$}, otherwise we let
$x_1=x$ {and $r_x=\delta_D(x)$}.
Denote $D_x=B(x_1,{r_x})$.
Similarly, we let $D_y=B(y_1,{r_y})$, where
$y_1 = y_0(1-r/2)+z_2$ if  $\delta_D(y)<r/6$ and $r_y=r/2$, with  $y_0=(y-z_2)/|y-z_2|$, and we let $y_1=y$, $r_y=\delta_D(y)$ otherwise.

CASE I.
We first assume that $|x-y|>2r$.
For $u\in D_x$ and  $v\in D_y$ we have
$$ |u-v| \le |u-x_1|+|x_1-x|+|x-y|+ |y-y_1|+|y_1-v|
\le  |x-y|+ 2r \le 2|x-y|. $$
 We next use  Lemma \ref{lemppu100}
  with  $D_1= D_x$ and $D_3=D_y$, and obtain
	\begin{eqnarray*}p_{D} (t, x , y) &\ge&
    \p^{x}(\tau_{D_1}>t)\p^{y}(\tau_{D_3}>t)\,\;
    t\!\!\!\!\!\inf_{u\in D_1,\, v\in D_3}\!\!\!\nu(u-v)\\
    &\ge&   \p^{x}(\tau_{D_1}>t)\p^{y}(\tau_{D_3}>t)\, t\,\nu(2|x-y|\wedge \diam(D)).
		\end{eqnarray*}
By subadditivity of $V$  we have $t= c^*V^2(r) \le {\Cb} V^2(r)/{36}\le  \Cb V^2(r/6)\le  \Cb V^2({r_x})$. By Lemma \ref{L5a1} and
\ref{Levy_V},
\begin{eqnarray*}
\p^x(\tau_{D_1}>t)&\ge&  \Cc \frac{c_1}{{H_1
} \Cl }  \left(\frac{V({\delta_{D_x}(x)})}{\sqrt{t}}\wedge1\right)= \Cc \frac{c_1}{ H_{{1}}
\Cl }  \left(\frac{V(\delta_D(x))}{\sqrt{t}}\wedge1\right),
\end{eqnarray*}
 where  $c_1=c_1(d)$.
Hence,
 \begin{eqnarray*}p_{D} (t, x , y) &\ge&c_2(H_{1}\Cl )^{-2} \left(\frac{V(\delta_D(x))}{\sqrt{t}}\wedge1\right)\left(\frac{V(\delta_D(y))}{\sqrt{t}}\wedge1\right)
t\nu(2|x-y|\wedge \diam(D))
,
\end{eqnarray*}
where $c_2=c_2(d)$. Since  $\Cl \ge 1$ and $\lC\le 1$, we have a complete proof in this case.

{CASE II. $x,y\in D: |x-y|\le 2r$. We define $\tilde{D}_x=B(x_1, r/{12})$.

		Since $ t= c^*_0V^2(r)\le \Cd V^2(r)$, by Corollary \ref{int2}, Lemma \ref{Levy_V} {and \eqref{Levy-V1}} we have,
		 \begin{eqnarray} \int_{\tilde{D}_x}p_{D}(t/3, x, v)dv &\ge& \Ce\frac{c_3}{{H_{1}
} \Cl }   c^*V^2(r) \nu(r)r^d\left(\frac{V(\delta_D(x))}{\sqrt{t}}\wedge1\right)\nonumber\\
			&\ge&  \frac {c_4} {{H_
{1}}( \Cl )^{{3}}
} \left(\frac{V(\delta_D(x))}{\sqrt{t}}\wedge1\right), \label{int_estimate}
		\end{eqnarray}
where  $c_4=c_4(d)$. A similar inequality obtains for $ \int_{\tilde{D}_y}p_{D}(t{/3}, y, v)dv$.}

Let $u\in \tilde{D}_x$ and $	v\in \tilde{D}_y$.
We claim that there is $  c_5=c_5(d, \la) $ such that
\begin{equation}\label{oodpuv}
p_{ D}(t/3, u,v)\ge c_5 \frac{ {\lC^{1+d/\la}}}  {(\Cl )^{3+d/\la}} p_t(0).
\end{equation}
Indeed, we
have   $|u -v |\le  3r$. Our aim is to estimate  $\E^{u} p(t/3-\tau_{ D},X_{\tau_{ D}}, v)$.  Since $|z-v|\ge r/{12}$ for all $z\in D^c$, by (\ref{B}) and subadditivity of $V$ we obtain
	 \begin{eqnarray*}\E^{u} p(t/3-\tau_{ D},X_{\tau_{ D}}, v)&\le& 12^d{\Cf}\frac{t}{V^2(r/12)r^d}\p^{u}(\tau_{ D}\le t/3) \leq  12^{d+2}{\Cf}\frac{t}{V^2(r)r^d}\p^{0}(\tau_{ B_{r/12}}\le t/3)\\
	&\le& \frac{12^{d+4}\Cj\Cf}{r^d}\left(\frac{t}{V^2(r)}\right)^2, \end{eqnarray*}
	where the last step uses \eqref{eq:2}.
	Next, since $ t\le V^2(3r)$  {and $r\leq 1/3$},  by (\ref{A}) we have
	$$p(t/3,u,v)\ge \frac{t}{3\Cl V^2(3r)(3r)^d}\ge  \frac{t}{3^{{d+3}}\Cl V^2(r)r^d}.$$
	Recall that $\psi\in\WLSC{\la}{1}{\lC}$,
$t=c^*V^2(r)\le c^*V^2(1)$ and
 $c^*\Cl {\Cj\Cf 12^{d+4}}3^{d+3}\le1/2$.
Thus,
\begin{eqnarray*}p_{ D}(t/3, u,v)&=& p(t/3, u,v)- \E^{u} p(t/3-\tau_{ D},X_{\tau_{ D}}, v) \\
&\ge& c^*\frac1{3^{d+3}\Cl r^d} - (c^*)^2\frac{12^{d+4}\Cj\Cf}{r^d}\\
&\ge& \frac {c^*}{2\Cl 3^{d+3}}\frac1{r^d}=\frac {c^*}{2\Cl 3^{d+3}}(V^{-1}(\sqrt {t/c^*}))^{-d} ,\end{eqnarray*}
and by \eqref{sup-p-t-1},
$$\left(V^{-1}\left(\sqrt {t/c^*}\right)\right)^{-d}\ge  c_6 (\lC c^*)^{1+d/\la}p_t(0),$$
	 with  $ c_6=c_6(d,\la)$.  Since $c^*\Cl $ is a positive constant depending only on $d$, we obtain \eqref{oodpuv}.
By \eqref{oodpuv} and (\ref{int_estimate}),
		 \begin{eqnarray*}
    p_{D}(t,x,y)&\geq& \int_{\tilde{B}_y}\int_{\tilde{B}_x}   p_{D}(t/3,x,u)p_{D}(t/3,u,v)p_D(t/3,v,y)dudv\\
    &\geq&c_5 \frac{ \lC^{1+d/\la}}  {(\Cl )^{3+d/\la}}\; p_t(0)\int_{\tilde{D}_x}   p_{D}(t/3,x,u)du\int_{\tilde{D}_y}p_D(t/3,v,y)dv\\
 &\geq&   c_7 \frac{\lC^{1+d/\la}}{
{H_1^2}(\Cl )^{
{9}+d/\la}}   \left(\frac{V(\delta_D(x))}{\sqrt{t}}\wedge1\right) p_t(0) \left(\frac{V(\delta_D(y))}{\sqrt{t}}\wedge1\right),
  \end{eqnarray*}
	with  $  c_7=c_7(d, \la) $. The proof is complete, cf. Remark~\ref{rem:sVi}.
		\end{proof}
The end of the above proof shows a major strategy for $X$ to connect $x$ and $y$  and survive in $D$ time $t$: evade $D^c$ by going from $x$ and $y$ towards the center of $D$ and connect then.

\begin{rem}\label{lower_scaling1}  Suppose that   global WLSC and WUSC hold for $\psi$. Then (\ref{A}) holds for all $t$ and $x$ such that $  {0<}t\le V^2(|x|) $ with the constant $\Cl $ depending only on $d$ and $\psi$. Furthermore, $\As$ holds. It follows
that the constant in the lower bound in Theorem~\ref{lower_hk_estimate}
may be so chosen
to
depend only on $d$ and $\psi$.
\end{rem}

\begin{cor}\label{lower_scaling}
Suppose that   global {\rm WLSC} and {\rm WUSC} hold for $\psi$.
Constants $c^*=c^*(d,\psi)$ and $C^*=C^*(d,\psi)$
exist such that for every open  $D$  with  inner ball condition at scale $R$,
$$p_{D}(t,x,y)\ge C^*  \left(\frac{V(\delta_D(x))}{\sqrt{t}}\wedge1\right)\left(\frac{V(\delta_D(y))}{\sqrt{t}}\wedge1\right)p(t,x,y),$$
if $0<t\le c^* V^2(R)$.
If $D=\H_0$ or $D=\overline{B}_{R}^c$, then the estimate is true for all $t>0$.
\end{cor}
\begin{proof} This easily follows from domain monotonicity by using a union of two balls of radius $R$ instead of $D$ and applying Theorem \ref{lower_hk_estimate} along with Remark \ref{lower_scaling1} and \cite[Corollary 23]{2014-KB-TG-MR-jfa}.
\end{proof}

{The following improvement of \cite[Theorem 5.10]{2012arXiv1212.3092K} stems from Corollary \ref{lower_scaling} and Theorem \ref{heatKernelComplGlobal}.
\begin{cor}\label{hk_Halfspace_Approx} If global {\rm WLSC} and {\rm WUSC} hold for $\psi$,
then for  $D=\H_0$, $x,y\in D$ and $t>0$,
$$p_{D}(t,x,y)\approx  \left(\frac{V(\delta_D(x))}{\sqrt{t}}\wedge1\right)\left(\frac{V(\delta_D(y))}{\sqrt{t}}\wedge1\right) p(t,x,y),$$
with comparability constant depending only on $d$ and the scaling characteristics of $\psi$.
\end{cor}
Below we give sharp heat kernel estimates for other classes of $C^{1,1}$ sets.
}		

\section{Global estimates for bounded $C^{1,1}$ sets}\label{sec:uni}

In this section we provide sharp explicit estimates of the heat kernel of bounded $C^{1,1}$ open sets.
To this end we combine
spectral properties of the heat kernel $p_D(t,x,y)$ for large time with the finite-time estimates
obtained in Sections~\ref{sec:UpperBound} and \ref{sec:LowerBound}.
Our discussion of spectral properties of $p_D$ closely follows that in \cite[the proof of Theorem~1.1]{MR2981852} but we additionally provide explicit control of comparability constants, which is delicate in the intermediate region between small and large times.
For instance under global scaling conditions on $\psi$ we give
an estimate of the heat kernel of the ball of
arbitrary radius, uniform enough to reproduce
optimal estimates of
the  heat kernel of a halfspace.

 Let $D$ be an open bounded set. In the remainder of the section we  assume that $p_t(0)$ is finite for every $t>0$, cf. Lemma~\ref{HW}. Then  the semigroup of integral operators on $L^2(D)$
with kernels $p_{D}(t, x, y)\le p_t(0)$ is compact, in fact Hilbert-Schmidt.
General theory yields
eigenvalues
 $0<\lambda_1<\lambda_2\le \dots$  and
orthonormal basis of eigenfunctions $\phi_1\ge 0,\phi_2, \phi_3 \dots$:
$$\phi_k(x) = e^{\lambda_{k}t} \int p_D(t,x,z)\phi_k(z)dz.$$

\begin{lem}\label{L2}  Let   $f\in L^2(D)$. Then
\begin{equation*}\label{eigen20} e^{-\lambda_1t}\left(\int f(w)\phi_1(w)dw\right)^2 \le \int \int f(w) f(z) p_D(t,z,w)dzdw \le e^{-\lambda_1t}\int f^2(w)dw. \end{equation*}
\end{lem}

\begin{proof}
The result obtains from the identities
\begin{eqnarray*}\int f^2(w)dw &=& \sum_k\left(\int f(w)\phi_k(w)dw\right)^2,\\
\int \int f(w) f(z) p_D(t,z,w)dzdw&=&\sum_k e^{-\lambda_k t}\left(\int f(w)\phi_k(w)dw\right)^2.
\end{eqnarray*}
\end{proof}

The following general bound is an easy consequence of Lemma \ref{L2}.
\begin{lem} \label{UB1}Let $t_0>0$. For $t\ge t_0$ and $x,y\in D$,
 \begin{eqnarray*}p_D\(t,x,y\) \le |D|\,p_{t_0/4}(0)^2 \, \p^x\(\tau_{D}>\frac{t_0}4\) \p^y\(\tau_{D}>\frac{t_0}4\)  e^{\lambda_1t_0}e^{-\lambda_1t}. \end{eqnarray*}
  \end{lem}
\begin{proof}
Let $t_0>0$. We use Lemma \ref{UB} and    Lemma \ref{L2} with  $f\equiv I_D$. Then for $t>t_0$,
\begin{eqnarray*}
p_D(t,x,y)&=&\int \int p_D\(\frac{t_0}{2},x,z\)p_D\(t-t_0,z,w\)p_D\(\frac{t_0}{2},w,y\)dzdw\\
&\le& [p_{t_0/4}(0)]^2 \p^x\(\tau_{D}>\frac{t_0}4\) \p^y\(\tau_{D}>\frac{t_0}4\) \int \int  p_D\(t-t_0,z,w\)dzdw\\
&\le& |D|\,p_{t_0/4}(0)^2\, \p^x\(\tau_{D}>\frac{t_0}4\) \p^y\(\tau_{D}>\frac{t_0}4\)  e^{\lambda_1t_0}e^{-\lambda_1t}.\\
\end{eqnarray*}
\end{proof}
We now discuss the corresponding lower bound.
\begin{lem} \label{LB1} If $t_0>0$ and
$c_*>0$ are such that
\begin{equation} \label{lower_100}
p_D\left(\frac{t_0}{2},x,z\right)
\ge c_*   \p^x\left(\tau_{D}>\frac{t_0}{2}\right) \p^z\left(\tau_{D}>\frac{t_0}{2}\right), \quad x, z \in D,
\end{equation}
then for $t\ge t_0$ and $x,y\in D$,
$$p_D(t,x,y)\ge \(\frac{c_*}{\sqrt{|D|}p_{t_0/2}(0)}\)^2e^{-\lambda_1t_0} \p^x\(\tau_{D}>\frac{t_0}{2}\) \p^y\(\tau_{D}>\frac{t_0}{2}\)
  e^{-\lambda_1t}.$$
  \end{lem}
\begin{proof}
Since $\lambda_1$ is the eigenvalue corresponding to $\phi_1$,

$$\phi_1(x) = e^{2\lambda_1s} \int p_D(2s,x,z)\phi_1(z)dz.$$
  From Lemma \ref{UB}, $p_D(2s,x,z) \le \p^x(\tau_{D}>s) p_s(0)$, and by Schwartz inequality,
\begin{eqnarray*}\phi_1(x) &\le&  e^{2\lambda_1s}  p_s(0) \p^x(\tau_{D}>s) \int \phi_1(z)dz
\leq \sqrt{|D|}\ e^{2\lambda_1s}  \p^x(\tau_{D}>s)  p_s(0).\end{eqnarray*}
Taking $s= {t_0}/2$,
we obtain
$$\phi_1(x) \le\sqrt{|D|} \, p_{t_0/2}(0) \ e^{\lambda_1t_0} \p^x\left(\tau_{D}>\frac{t_0}2\right),   $$
which in turn yields
\begin{equation} \label{eigen10}1=\int \phi_1^2(z)dz \le  \sqrt{|D|}\, p_{t_0/2}(0) \ e^{\lambda_1t_0}\int \phi_1(z) \p^z\left(\tau_{D}>\frac{t_0}2\right)dz.   \end{equation}

Let $t> t_0$. By (\ref{lower_100}), Lemma \ref{L2} with $f(z)=  \p^z(\tau_{D}>{t_0}/{2})$  and  (\ref{eigen10}) we have
\begin{eqnarray*}
p_D(t,x,y)&=&\int \int p_D\(\frac{t_0}{2},x,z\)p_D\(t-t_0,z,w\)p_D\(\frac{t_0}{2},w,y\)dzdw\\
&\geq&c_*^2 \p^x\(\tau_{D}>\frac{t_0}{2}\) \p^y\(\tau_{D}>\frac{t_0}{2}\)\\
 &&\times \int \int \p^z\(\tau_{D}>\frac{t_0}{2}\) \p^w\(\tau_{D}>\frac{t_0}{2}\)  p_D\(t-t_0,z,w\)dzdw\\
&\ge & c_*^2 \p^x\(\tau_{D}>\frac{t_0}{2}\) \p^y\(\tau_{D}>\frac{t_0}{2}\)
  e^{-\lambda_1\(t-t_0\)} \left(\int \p^w\(\tau_{D}>\frac{t_0}{2}\) \phi_1(w) dw\right)^2\\
&\ge &\(\frac{c_*}{\sqrt{|D|}p_{t_0/2}(0)}\)^2e^{-\lambda_1t_0} \p^x\(\tau_{D}>\frac{t_0}{2}\) \p^y\(\tau_{D}>\frac{t_0}{2}\)
  e^{-\lambda_1t}. \\
\end{eqnarray*}
\end{proof}
\noindent
Lemma~\ref{UB1} and \ref{LB1} indicate the asymptotics of the killed semigroup for large times.

In what follows, we interchangeably write $\lambda_1(D)=\lambda_1$. Here is a sharp bound for the first eigenvalue in terms of $V$.
\begin{prop} \label{eigenvApprox}
Let $D$ be an open bounded set containing a ball of radius $r$. Then
 $$\frac 1{8} \left(\frac r  {\diam D} \right)^{2} \le  \lambda_1(D)V^2(r) \le  c \left(\frac {\diam D} r\right)^{d/2},$$
where  $c=c(d)$.
\end{prop}
\begin{proof}
The following bound is proved in \cite[Proposition 2.1]{BanKul}:
$$\frac{ 1}{ \sup_x \E^x\tau_D }\le \lambda_1(D)\le \frac{ \int_D \E^x\tau_D dx}{ \int_D (\E^x\tau_D)^2 dx}. $$
By the  Cauchy-Schwarz inequality,
$$\frac{ \int_D \E^x\tau_D dx}{ \int_D (\E^x\tau_D)^2 dx}\le \sqrt{ \frac {|D|}{ \int_D (\E^x\tau_D)^2 dx}}. $$
Let $B(x_0,r)\subset D$.
By \cite[Lemma 2.3]{BGR2},
$\sup_x \E^x\tau_D \le 2V^2(\diam D)$ and by \eqref{eq:l}, \newline $\inf_{x\in B(x_0,r/2)} \E^x\tau_D \ge V^2(r)/\Ci$. Hence,
$$
\lambda_1(D)\le \Ci\frac{ {1}}{V^2(r)} \sqrt{\frac{ |D|}{ |B(x_0,r/2)|}}\le
\frac {\Ci 2^{d/2}}{V^2(r)} \left(\frac {\diam D} r\right)^{d/2} $$
and
$$ \lambda_1(D)\ge \frac 1{2V^2(\diam D)}= \frac1 {2V^2(r)} \frac{ V^2(r)}{V^2(\diam D)}\ge \frac1 {8V^2(r)} \left(\frac{ r}{\diam D}\right)^2, $$
where in the last step we used subadditivity of $V$.
\end{proof}

Here is the main result of this section (cf. Remark \ref{ScAgree} for our notational conventions).
\begin{thm}\label{hkC11_gen} Let $\psi\in {\rm WLSC }\cap{\rm WUSC}$. There is $r_0= r_0(d,\psi)>0$ such that if $0<r< r_0$ and  open $D\subset \Rd$ is bounded and $C^{1,1}$  at scale $r$, and $\nu(\diam D)>0$, then
for all $x,y\in \Rd$, $t>0$,
\begin{equation}
p_D(t,x,y)\approx \p^x(\tau_D>t/2)\p^y(\tau_D>t/2)p\(t\wedge V^2(r),x,y\)
\end{equation}
and
\begin{equation}\label{sbgenC11W}
\p^x\(\tau_{D}>t\)\approx  e^{-\lambda_1 t} \(\frac{V(\delta_D(x))}{\sqrt{t}\wedge V(r)}\wedge 1\).
\end{equation}
If the scalings
are global, then we may take $r_0=\infty$ and
comparability constants depending only on $d$, $\diam D/r$
and  scaling characteristics of $\psi$.
\end{thm}

\begin{proof}
Define $t_0=V^2(r)$, and for $x\in
\Rd,\ t>0, \ s>0$,
 $$\tilde{p}_t(x)=
p_{t/2}(0) \wedge \frac {t}{V^2(|x|) |x|^d},$$
$$\hat{p}_t(x)=   p_{t/2}(0)\wedge [t\,\nu(|2x|\wedge \diam(D))],$$
$$\phi(t,x,s)= \frac{V(\delta_D(x))}{\sqrt{t}\wedge V(s)}\wedge1,$$
with the convention that $\tilde{p}_t(0)=p_{t/2}(0)=\hat{p}_t(0)$.
Clearly, $\tilde{p}_t(x)$ and $\hat{p}_t(x)$ are nonincreasing functions of $|x|$, $\phi(t,x,s)$ is nonincreasing in $t$,
$$\hat{p}_t(x)/\tilde{p}_t(x)=   p_{t/2}(0)/[\frac {t}{V^2(|x|) |x|^d}]\wedge [\nu(|2x|\wedge \diam(D))]/[{1}/{V^2(|x|) |x|^d}],$$
$\hat{p}_t(x)/\hat{p}_t(y)\ge \nu(\diam(D))/\nu(|2y|\wedge \diam(D))$, and $k\hat{p}_{t}(x)\ge \hat{p}_{kt}(x)$ for $k\ge 1$.
Since $\psi\in {\rm WLSC }\cap{\rm WUSC}$, by Lemma~\ref{densityApprox} and Lemma~\ref{sup_p_t},
\begin{equation}\label{LUB}c^{*} \tilde{p}_t(x)  \le p_t(x)   \le \Cf \tilde{p}_t(x), \end{equation}
where $c^*=c^*(d,\psi)\le 1$, $r_0=r_0(d,\psi)$,
$|x|<r_0$,
$0<t<  t_1$
and $t_1= V^2(r_0)$. The upper bound in \eqref{LUB} even holds for all $t>0$, $x\in \Rd$. In particular, \eqref{A}, in fact ${\bf G}_{r_0}$, hold with $\Cl =c(d)/c^*$.
By letting $t\to 0$ in \eqref{LUB}, we get for $|x|<r_0$,
 \begin{equation}\label{LUBml}\frac {c^* }{V^2(|x|) |x|^d}  \le \nu(x)   \le  \frac {\Cf  }{V^2(|x|) |x|^d}. \end{equation}

If
$r_0\le \diam D$, then we can extend the lower bound in \eqref{LUB} and \eqref{LUBml} to $r_0\le |x|\le \diam D$ and $0<t\leq t_0$ due to Lemma \ref{Gr_Ex}.
 For $s=|2x|\wedge\diam D$, by \eqref{LUBml} and subadditivity of $V$,
$$\nu(s)\geq \frac{c^*}{V^2(s)s^d}\geq \frac{c^*}{V^2(2|x|)2^d|x|^d}\geq  \frac{c^*}{2^{d+2}}\frac{1}{V^2(|x|)|x|^d}.$$
By the definitions of $\hat p$ and $ \tilde p$, for $x,y\in D$ and $0<t\leq t_0$ we have  \begin{equation}\label{comp10}\hat{p}_t(x-y)\geq  \frac{c^*}{2^{d+2}}\tilde{p}_{t}(x-y)\geq  \frac{c^*}{2^{d+2}\Cf}p(t,x,y).\end{equation}

From now on we let $t>0$ and $x,y \in D$
(additional restrictions are indicated as we proceed).
We may  assume that
$\psi\in\WLSC{\la}{ \lt}{\lC}$ with  ${ \lt}= 1/r$,  cf. Remark~\ref{rem:ps}.
By Corollary \ref{heatKernelCompl1},
\begin{equation}\label{UB10} p_{D}(t, x, y)
    \le  c_1  \phi(t,x,r)\phi(t,y,r)\tilde{p}_t(x-y),\qquad t\le t_0.
\end{equation}
where
$c_1=
c  H_{r}^2\mathcal{J}(r)^{-4} \lC^{-1-d/\alpha}$ and $c=c(d,
\la)$. To facilitate justification of the last statement of the theorem, all enumerated constants are fixed throughout the proof.
Let $\lambda_1= \lambda_1(D)$.
By \cite[Lemma 6.2]{BGR2} there is $c_2=c_2(d,r, \psi)$ such that $\p^x\(\tau_{D}>t/4\)\le \sqrt{c_2} \phi(t,x,r)$.
 By Lemma \ref{UB1}, for $t\ge t_0$
we have
 \begin{eqnarray}p_D\(t,x,y\) &\le& |D|\,p^2_{t_0/4}(0) \, \p^x\(\tau_{D}>\frac{t_0}4\) \p^y\(\tau_{D}>\frac{t_0}4\)  e^{\lambda_1t_0}e^{-\lambda_1t}\nonumber\\
 &\le& c_2 |D|\,p^2_{t_0/4}(0)\, \phi(t_0,x,r)\phi(t_0,y,r)   e^{\lambda_1t_0}e^{-\lambda_1t}\nonumber\\
 &\le& c_2 |D|\,\frac{p^2_{t_0/4}(0)}{p_{t_0}( \diam D) }e^{\lambda_1t_0} \, \phi(t_0,x,r)\phi(t_0,y,r)p({t_0},x,y)   e^{-\lambda_1t}.
  \label{UB11}\end{eqnarray}
 Combining \eqref{UB10} and \eqref{UB11} with \eqref{LUB}
 we get
   \begin{equation}p_D\(t,x,y\) \le c_3 \phi(t,x,r)\phi({t},y,r) p({t\wedge t_0},x,y)   e^{-\lambda_1t},\quad {t>0},\label{UB13} \end{equation}
 where $c_3= \left(c_2 |D|\,{p^2_{t_0/4}(0)}/p_{t_0}(\diam D)+c_1/c^*\right) e^{\lambda_1t_0}$.

We now give
a similar lower bound.
By Theorem \ref{lower_hk_estimate} and domain monotonicity of heat kernels,
there is $c=c(d, \psi)<1$ such that for $t\le t_2:=c t_0$,
\begin{equation}\label{LB10}p_{D}(t,x,y)\ge   c_4  \phi( t,x,r)\phi(t,y,r)
	 \hat{p}_t(x-y),\end{equation}
	 with $c_4={ c(\la,d) \lC^{1+d/\la}} {{H_{r}^{-2}}(\Cl )^{{-{9}-d/\la}}}$.	
Since
$\phi(t,x,r)\ge \p^x\(\tau_{D}>t \)/{\sqrt{c_2}}$, we have
$$ p_{D}(t,x,y)\ge   {c_5} \p^x\(\tau_{D}>t \)\p^y\(\tau_{D}>t \)
		\hat{p}_t(x-y), \qquad t\le t_2.$$
	where $c_5= {c_4}/{c_2}$.	
		In particular,
 \begin{eqnarray*} p_{D}(t_2/2,x,y)&\ge&  {c_5} \hat{p}_{t_2/2}(\diam D)  \p^x\(\tau_{D}>t_2/2 \)\p^y\(\tau_{D}>t_2/2 \).
		 \end{eqnarray*}
By Lemma~\ref{L5a1} there is $c_6=c_6(d,r,\psi)$ such that $\p^z\(\tau_{D}>t_2/2\)\ge\sqrt{c_6} \phi( t_2,z,r), \ z\in D$. By Lemma \ref{LB1} for $t\ge t_2$ we have
 \begin{eqnarray*}p_D(t,x,y)&\ge& \(\frac{ c_5 \hat{p}_{t_2/2}(\diam D)}{\sqrt{|D|}p_{t_2/2}(0)}\)^2e^{-\lambda_1t_2} \p^x\(\tau_{D}>\frac{t_2}{2}\) \p^y\(\tau_{D}>\frac{t_2}{2}\)
  e^{-\lambda_1t}\\&\ge& c_6 \(\frac{ c_5 \hat{p}_{t_2/2}(\diam D)}{\sqrt{|D|}p_{t_2/2}(0)}\)^2e^{-\lambda_1t_2} \phi( t_2,x,r)\phi(t_2,y,r)e^{-\lambda_1t}. \end{eqnarray*}
By the above-mentioned monotonicity properties of $\hat{p}$,
for $t\ge t_2$ we have
    \begin{equation}\label{LB11}p_D(t,x,y)\ge c_7 \hat{p}_{t_2}(x-y)\phi( t_2,x,r)\phi(t_2,y,r)
  e^{-\lambda_1t}, \end{equation}
  where
  $c_7=  c_6 { c_5^2 e^{-\lambda_1t_2}\hat{p}^2_{{t_2}/2}(\diam D)}/\[{p_{t_2/2}^2(0)|D|\hat{p}_{t_2}(0)}\]$. Combining \eqref{LB10} and  \eqref{LB11} we get
    \begin{equation}\label{LB11212}p_D(t,x,y)\ge (c_4\wedge c_7) \hat{p}_{t\wedge t_2}(x-y)\phi( t\wedge t_2,x,r)\phi(t\wedge t_2,y,r)
  e^{-\lambda_1t}\qquad (t>0).\end{equation}
Since $\hat{p}_{t\wedge t_2}(x-y)\ge ({t_2}/{t_0}) \hat{p}_{t\wedge t_0}(x-y)$,
by the aforementioned monotonicity of $\phi$ and \eqref{comp10},
     \begin{eqnarray}p_D(t,x,y)&\ge& c_8 \hat{p}_{t\wedge t_0}(x-y)\phi( t\wedge t_0,x,r)\phi(t\wedge t_0,y,r)
  e^{-\lambda_1t}\nonumber\\&\ge& c_8 c_9 {p}({t\wedge t_0},x,y)\phi( {t},x,r)\phi({t},y,r)
  e^{-\lambda_1t},\label{LB12}  \end{eqnarray}
  where   $c_8= ({t_2}/{t_0}) (c_4\wedge c_7)$, and  $c_9= {c^*/(2^{d+2}\Cf)}>0$.
 Combining \eqref{UB13} with \eqref{LB12} we get
\begin{equation}\label{hkC11dow}
p_D\(t,x,y\) \approx \phi(t,x,r)\phi(t,y,r) p(t\wedge t_0,x,y)   e^{-\lambda_1t}.
\end{equation}
 Now we prove \eqref{sbgenC11W}. The upper bound:
$\p^x(\tau_D>t)\leq c_3 e^{-\lambda_1 t}\phi(t,x,r)$ is an easy consequence of \eqref{hkC11dow} since $\phi(t,y,r)\leq 1$ and $\int_{\Rd} {p}(t\wedge t_0,x,y)dy=1$.   Lemma \ref{lem:spgenC11} implies the lower bound for $t\leq t_2$, cf. \eqref{eq:dcg}. If $t>t_2$ then, by \eqref{LB11212} and monotonicity of $\phi$,
$$\p^x(\tau_D>t)\geq (c_4\wedge c_7) e^{-\lambda_1 t} \phi(t_2,x,r)\hat{p}_{t_2}(\diam D)\int_D
\phi(t_2,y,r)dy\geq c_{10} e^{-\lambda_1 t} \phi(t,x,r) ,$$
where $c_{10}=(c_4\wedge c_7)  \hat{p}_{t_2}(\diam D)\int_D
\phi(t_2,y,r)dy$.

	 We now assume that $\psi$ satisfies global scaling conditions.
Then $t_1=r_0=\infty$.
We shall investigate the dependence of the constants
$c_1$-$c_{10}$
on $r$ and  $\diam D$. The constants $c_1$, $c_2$, $c_4$-$c_6$, $c_9$ depend only on  the dimension and $\psi$ (through scaling characteristics), but not on $r$ or $\diam D$. This is due to  \cite[Proposition 5.2, Lemma~7.2 and Lemma~7.3]{BGR2} and \cite[Proposition~3.5]{2012GR}, which imply that  quantities   $\mathcal{J}(s)$,  $\mathcal{I}(s)$ and $H_s$ are uniformly bounded in $s\in(0, \infty)$  from below and above  by two positive constants.    {Furthermore, $c_8$ depends only on $c_7$,  $d$ and the scaling characteristics}. Therefore we only  need to inspect {$c_3$, $c_7$ and $c_{10}$.}
We claim that $c_3\le c^*_3=c^*_3({\diam D}/{r}, d,\psi)<\infty$,     $c_7\ge c^*_7=c^*_7({\diam D}/{r},d, \psi)>0$ {and $c_{10}\ge c^*_{10}=c^*_{10}({\diam D}/{r},d, \psi)>0$}.
The remaining comparisons in this proof depend only on $\psi$ and $d$.
We have $t_0/4\approx  t_2/2 \approx V^2(r)$, then $ p_{t_0/4}(0)\approx p_{t_2/2}(0) {=} \hat{p}_{t_2}(0)\approx r^{-d}$. Furthermore,
$$\tilde{p}_{t_0}(\diam D)\approx  \hat{p}_{\frac {t_2}2}(\diam D)\approx  {\hat{p}_{t_2}(\diam D)\approx}{p}_{t_0}(\diam D)\approx \frac {V^2(r)}{V^2(\diam D) ({ \diam D})^d},$$
and
$$\frac{|D|p_{t_0/4}(0)^2}{\tilde{p}_{t_0}( \diam D)}\approx  \frac {|D| (\diam D)^dV^2(\diam D)  }{V^2(r)r^{2d}}\le c \frac { (\diam D)^{2d}V^2(\diam D)  }{V^2(r)r^{2d}}\le c \(\frac {\diam D}{r}\)^{2d+2},$$
where in the last step we used subadditivity {\eqref{subad}} of $V$, and $c=c(d)$. By the same arguments,
$$\(\frac{  \hat{p}_{{t_2}/2}(\diam D)}{p_{t_2/2}(0)\sqrt{|D|\hat{p}_{t_2}({0})}}\)^2\ge  c^{-1} \(\frac {r} {\diam D}\)^{{3d+4}}.$$
By Proposition \ref{eigenvApprox} we have
$$  \lambda_1t_2 \le \lambda_1t_0 = \lambda_1 V^2(r)  \le c  \left(\frac {\diam D} {r} \right)^{d/2},$$
where $c=c(d)$. Furthermore,
 $$ \int_D \phi( t_0 ,y,r)dy\ge \int_{B_{r/2}} \(\frac{V(r/2)}{ V(r)}\wedge 1\)dy \ge  |B_{r/2}|/2,$$
{hence
$$\hat{p}_{t_2}(\diam D)\int_D \phi( t_0 ,y,r)dy\geq c \frac {V^2(r)r^d}{V^2(\diam D) ({ \diam D})^d}\geq c\(\frac{r}{\diam D}\)^{d+2}.$$}
The above arguments indeed  show that if the   global scaling {conditions  hold}, then the constants $c_1$-$c_{10}$ depend only through ${\diam D}/ {r}$,  $d$ { and the scaling characteristics of $\psi$}.

\end{proof}
Here is a simple consequence of Theorem~\ref{hkC11_gen}.
\begin{cor}\label{hk_kula2}If $D=B_{R}$, $\lambda_1(R)=\lambda_1(D)$   and
$\psi\in\WLSC{\la}{ 0}{\lC}\cap\WUSC{\ua}{ 0}{\uC}$,
 then
 \begin{eqnarray*} p_{D}(t, x, y)
    &\approx&  e^{-\lambda_1(R) t}\(\frac{V(\delta_D(x))}{\sqrt{t}\wedge V(R)}\wedge 1\)\(\frac{V(\delta_D(y))}{\sqrt{t}\wedge V(R)}\wedge 1\)p(t\wedge V^2(R), x,y)\\
   &\approx&  \p^x\(\tau_{D}>\frac{t}{2}\)\p^x\(\tau_{D}>\frac{t}{2}\) p(t\wedge V^2(R), x, y),\end{eqnarray*}
for all $x,y\in \R^d$ and $t>0$. The comparability  constants depend only on  $d$ and the scaling characteristics of $\psi$.
\end{cor}

Such uniform estimates should be useful in approximation and scaling arguments, especially that $p_D$ is monotone in $D$.

\begin{cor}
Under the assumptions of Corollary \ref{hk_kula2}, let $\phi^R_1$ be the (positive) eigenfunction corresponding to
$\lambda_1(R)$.
 There is $c=c(d,\psi)$ such that
$$c^{-1}\frac{V(\delta_D(x))}{R^{d/2} V(R)}\leq \phi^R_1(x) \leq c \frac{V(\delta_D(x))}{R^{d/2} V(R)}, \qquad x \in \Rd.$$
\end{cor}
\begin{proof}
By Corollary  \ref{hk_kula2}, for $t\ge V(R)$,  we obtain
$$ p_{D}(t, x, x)
    \approx  e^{-\lambda_1(R) t}\(\frac{V(\delta_D(x))}{ V(R)}\)^2 p_{V^2(R)}(0).$$
By Lemma \ref{sup_p_t},  $$ p_{V^2(R)}(0)\approx  R^{-d}.$$
Since $\nu$ is radial and infinite, by \cite[Theorem 3.1]{MR2445505}, the semigroup $P^D_t$ is intrinsically ultracontractive. Hence, by \cite[Theorem 4.2.5]{MR990239}
$$\lim_{t\to\infty}\frac{p_D(t,x,x)}{e^{-\lambda_1(R)t}(\phi^R_1(x))^2}=1, \quad x\in D,$$
which gives the claimed result.

\end{proof}

\section{Unbounded sets}\label{sec:ud}

Throughout this section  we assume that global WLSC and WUSC hold for $\psi$. Due to  \cite[Proposition 5.2, Lemma~7.2 and Lemma~7.3]{BGR2} and \cite[Proposition~3.5]{2012GR}
the quantities   $\mathcal{J}(r)$,  $\mathcal{I}(r)$ and $H_r$ employed above
now depend only on the dimension and scaling characteristics of the L\'evy-Kchintchine exponent $\psi$.
Denote $L(r)=\nu(B_r^c)$, $r>0$, the tail of the L\'evy measure.
Our first unbounded set is the complement of a ball.
\begin{prop}\label{HittingDensity} Let $\psi\in
\WLSC{\la}{0}{\lC}
\cap \WUSC{\ua}{0}{\uC}$.
There
is $\Cg=\Cg (d,\psi)$, such that  for all $R>0$ and  $t\geq V^2(R)$,
$$\p^x(\tau_{\overline{B}^c_R}\in dt)/dt\leq \Cg p_{t/2}(0)\frac{R^d}{V^2(R)},\quad {x\in\Rd}.$$
\end{prop}
\begin{proof}
By
(\ref{Ikeda-Watanabe1}) and symmetry of $\nu$, the Radon-Nikodym derivative satisfies
$$\p^x(\tau_{\overline{B}^c_R}\in dt)/dt=\int_{\overline{B}^c_R} p_{\overline{B}^c_R}(t,x,y)\nu\(B(y,R)\)dy.$$
Since $\nu$ is radially decreasing, for $|y|>R$,
 $$\nu\(B(y,R)\)\leq \min\{L(|y|-R),\frac{\omega_d}{d}\nu(|y|-R)){R^d}\}.$$
By Lemma \ref{UB} and   \cite[Theorem 6.3]{BGR2}, there exists $c_1=c_1(d,\psi)$ such that, for  $t\geq V^2(R)$,
$$p_{\overline{B}^c_R}(t,x,y)\leq c_1 p_{t/2}(0)\(1\wedge \frac{V(|y|-R)}{V(R)}\), \qquad x,\, y\in \Rd.$$
Therefore,
\begin{eqnarray*}&&\p^x(\tau_{\overline{B}^c_R}\in dt)/dt
\leq c_1 p_{t/2}(0)\int_{\overline{B}^c_R} \(1\wedge\frac{V(|y|-R)}{V(R)}\)\nu\(B(y,R)\)dy\\
&&\leq c_1 p_{t/2}(0)\(\int_{R< |y|< 2R}\frac{V(|y|-R)}{V(R)}L(|y|-R)dy+\frac{\omega_d}{d}R^d\int_{|y|\geq2R}\nu(|y|-R))dy\right)\\
&&\leq\omega_d2^{d-1}c_1p_{t/2}(0)\(\frac{R^{d-1}}{V(R)}\int^R_0V(\rho)L(\rho)d\rho+R^dL(R)  \)\\
&& \leq c_2 p_{t/2}(0)\frac{R^d}{V^2(R)},
\end{eqnarray*}
where in the last line we used Lemma~\ref{ch1V} and   \cite[Proof of Proposition 3.4]{BGR2} to get $$\int^R_0V(\rho)L(\rho)d\rho\le c \frac { R} {V(R)}.$$
In fact, $c_2=c_1c(d).$
\end{proof}

The assumption $d>\la$ in the next result secures the {\it transience} of the underlying unimodal L\'evy process $X$ \cite[Corollary~37.6]{MR1739520}. The proof below asserts in relative terms that hitting the ball $\overline{B_R}$ is unlikely for $X$ when the points $x$ and $y$ are far away from the ball.

\begin{prop} \label{heatKernelOutside}
 Let $\psi \in \WLSC{\la}{0}{\lC}\cap \WUSC{\ua}{0}{\uC}$ and $d>\ua$.
There is $\Ch=\Ch (d,\psi)$ such that  if $R>0$,  $|x|,|y|\geq \Ch R$, and $t>0$, then
$$p_{\overline{B}^c_R}(t,x,y)\geq \frac{1}{2} p(t,x,y).$$
\end{prop}
\begin{proof}
Assume that $|y|\geq |x|\geq 2R$. Let $y^*$ be a projection of $y$ onto the boundary of $B_R$.
Let $f(t,x,y)=E^x[p(t-\tau_{\overline{B}^c_R},y^*,y)
,\tau_{\overline{B}^c_R}<t]$.
  Since $p_t(\cdot)$
is radially decreasing,
$$p_{\overline{B}^c_R}(t,x,y)=  p(t,x,y)- E^x[p(t-\tau_{\overline{B}^c_R},X_{\tau_{\overline{B}^c_R}},y),\tau_{\overline{B}^c_R}<t]\geq p(t,x,y)-f(t,x,y).$$
Clearly,
$$
f(t,x,y)\leq \sup_{s\leq t}p(s,y^*,y)\p^x(\tau_{\overline{B}^c_R}<\infty).
$$
Observe that $|y-y^*|\geq R\vee \frac{|x-y|}{4}$. Indeed
$$|x-y|\leq |x|+|y|\leq 2|y|\leq 4(|y|-R)=4|y-y^*|.$$
Due to Lemma~\ref{upper_den} and radial monotonicity of $p_t$,
\begin{equation*}\sup_{s\leq t}p(s,y^*,y)\leq c_1\frac{t}{V^2\(R\vee \frac{|x-y|}{4}\)\left|R\vee \frac{|x-y|}{4}\right|^d}.\end{equation*}
This, Lemma~\ref{densityApprox}, radial monotonicity of $p_t$ and \cite[Corollary 24]{2014-KB-TG-MR-jfa} give,
for $t\leq 2V^2\(R\vee \frac{|x-y|}{4}\)$,
$$\sup_{s\leq t}p(s,y^*,y)\le c_2 p_t\(R\vee \frac{|x-y|}{4}\)\leq c_3 p(t,x,y),$$
and so by
\cite[Proposition 5.8]{BGR2},
\begin{equation}\label{hKO1}
f(t,x,y)\leq c_4p(t,x,y)\frac{R^dV^2(|x|)}{V^2(R)|x|^d}.
\end{equation}
Let $t\geq 2V^2\(R\vee \frac{|x-y|}{4}\)$. By Proposition \ref{HittingDensity} and \cite[Theorem 3 and Section 4 for $d\leq 2$]{2013arXiv1301.2441G},
\begin{eqnarray*}
 \I&:=&E^x[p(t-\tau_{\overline{B}^c_R},y,y^*),t/2\leq \tau_{\overline{B}^c_R}<t]\leq \Cg p_{t/2}(0)\frac{R^d}{V^2(R)}\int^{t}_{t/2}p(t-s,y,y^*)ds\\
 &\leq& \Cg p_{t/2}(0)\frac{R^d}{V^2(R)}U(y-y^*)\leq c_{5} p_{t/2}(0)\frac{R^d}{V^2(R)}\frac{V^2(|y|-R)}{(|y|-R)^d},
\end{eqnarray*}
where $U(y)=\int_0^\infty p_t(y)dt\approx V^2(|y|)/|y|^d$ is the potential kernel of $X$ \cite{2013arXiv1301.2441G}.
By \cite[Proposition 5.8]{BGR2} and monotonicity of $s\mapsto p_s(0)$,
$$\II:=E^x[p(t-\tau_{\overline{B}^c_R},y^*,y),\tau_{\overline{B}^c_R}<t/2]\leq c_{6}p_{t/2}(0)\frac{R^dV^2(|x|)}{V^2(R)|x|^d}.$$
Since $t\geq V^2(|x-y|)/8$, by  Lemma~\ref{densityApprox},  $p_{t/2}(0)\le  c_7p(t,x,y)$. Therefore,
\begin{equation}\label{hKO2}
f(t,x,y)=\I+\II\leq c_8 p(t,x,y) \frac{R^d}{V^2(R)} \(\frac{V^2(|x|)}{|x|^d}+\frac{V^2(|y|-R)}{(|y|-R)^d}\).
\end{equation}
 Finally, combining (\ref{hKO1}) with (\ref{hKO2}), for all $t>0$ we have
\begin{eqnarray*}
f(t,x,y)&\leq&
(c_8+c_4)p(t,x,y)\frac{R^d}{V^2(R)} \(\frac{V^2(|x|)}{|x|^d}+\frac{V^2(|y|-R)}{(|y|-R)^d}\).
\end{eqnarray*}
By global WUSC, for all $t>0$ and $|x|,|y|\geq 2R$,
$$f(t,x,y)\leq c_9 p(t,x,y) \( \( \frac{R}{|x|}\)^{d-\ua}+\(\frac{R}{|y|-R}\)^{d-\ua}\).$$
Therefore there exists a constant $c_{10}$ such that  for $|x|,|y|\geq c_{10}R$ we have
$$p_{\overline{B}^c_R}(t,x,y)\geq p(t,x,y)-f(t,x,y) \geq\frac{1}{2} p(t,x,y).$$
\end{proof}

\begin{lem}\label{freeChange}Let $d\geq 1$ and $\psi$ satisfy global {\rm WLSC} and {\rm WUSC}.
There is a constant $C=C(d, \psi)$ such that  if $\lambda>1$ and $|x-z|\leq \lambda R$, then
 $$C^{-1}\lambda^{-2-d}\leq \frac{p_{V^2(R)}(x)}{p_{V^2(R)}(z)}\leq C\lambda^{2+d}.$$
\end{lem}
\begin{proof}Assume that $|x-z|\leq \lambda R$.
By symmetry it is enough to prove the upper bound.
By scaling and
Lemma~\ref{densityApprox}
we have, for  $x\in\R^d$,
\begin{equation}\label{eq:fC1}p_{V^2(R)}(x)\approx \min\left\{  R^{-d},\frac{ V^2(R))}{V^2(|x|)|x|^d}\right\}.\end{equation}
If $|z|\geq 2 \lambda R$, then $|z|\leq 2|x|$. Hence, by radial monotonicity and \cite[Corollary 24]{2014-KB-TG-MR-jfa},
$$p_t(x)\leq p_t(z/2)\leq c_1 p_t(z), \qquad
t>0.$$
For $|z|<2\lambda R$, again by radial monotonicity, $$p_t(z)\geq p_t(2\lambda R).$$
This, subadditivity of $V$, and (\ref{eq:fC1}) complete the proof: $$  \frac{p_{V^2(R)}(x)}{p_{V^2(R)}(z)}\leq \frac{p_{V^2(R)}(0)}{p_{V^2(R)}(2\lambda R)}\leq c_2\lambda^{2+d}.$$
\end{proof}

The next theorem may be considered as the main result of this section.
\begin{thm}\label{exterior_lower}  Let $\psi \in \WLSC{\la}{0}{\lC}\cap \WUSC{\ua}{0}{\uC}$ and $d>\ua$. {Let $D$ be  a $C^{1,1}$ at scale $R_1$ and $D^c\subset \overline{B_{R_2}}$.
Constants $c_*=c_*(d,\psi), c^*=c^*(d,\psi)$ exist such that for all $x,y\in\Rd$, $t>0$,
$$
    c_*\(\frac{R_1}{R_2}\)^{4+2d}  \(\frac{V(\delta_D(x))}{\sqrt{t}\wedge V(R_1)}\wedge 1\)\(\frac{V(\delta_D(y))}{\sqrt{t}\wedge V(R_1)}\wedge 1\)p(t, x, y) \le p_{D}(t, x, y) $$}
    and
   $$  p_{D}(t, x, y)\le
    c^*  \(\frac{V(\delta_D(x))}{\sqrt{t}\wedge V(R_1)}\wedge 1\)\(\frac{V(\delta_D(y))}{\sqrt{t}\wedge V(R_1)}\wedge 1\)p(t, x,y). $$
\end{thm}
\begin{proof} We only deal with the lower bound since the upper bound follows from Theorem \ref{heatKernelComplGlobal}. Assume that $|x|\leq |y|$ and denote $l(x,y)=\left(\frac{V(\delta_D(x))}{V(R_1)}\wedge 1\right)\left(\frac{V(\delta_D(y))}{V(R_1)}\wedge 1\right)$.
By Corollary \ref{lower_scaling} it is enough to consider $t>t_0=c^*V^2(R_1)$, where $c^*$ is the constant from that corollary. By Proposition \ref{heatKernelOutside},  for $|x|,|y|\geq \Ch R_2$,  we have
 \begin{equation}\label{LEB1}p_D(t,x,y)\geq p_{\overline{B}^c_{R_2}}(t,x,y)\geq\frac{1}{2}p(t,x,y).\end{equation}
In the remaining part of the proof we closely follow the ideas of \cite[Theorem 1.3]{MR2776619},
where a similar result is proved for the  isotropic stable L\'evy processes.
Let $|x|<\Ch R_2$. Fix $v\in \Rd$ with $|v|=1$, such that $\left<x,v\right>\geq 0$ and $\left<y,v\right>\geq 0$ if $d\geq 2$, and $v=2{y}/{|y|}$ if $d=1$. Define
$x_0=x+\Ch R_2 v$ and $y_0=y+\Ch R_2 v$. Then $|x_0|,\, |y_0|\geq \Ch R_2$.
{By  Lemma \ref{freeChange} and \cite[Corollary 24]{2014-KB-TG-MR-jfa}, there exists $c_1=c_1(d,\psi)$ such that, for all $z\in\Rd$,
$$p(t_0/2,x,z)\geq c_1(R_1/R_2)^{2+d}p(t_0/2,x_0,z).$$
Hence, by Corollary \ref{lower_scaling} and  Theorem  \ref{heatKernelComplGlobal}, there exists $c_2=c_2(d,\psi)$,  such that  for $z,x\in D$,
$$ p_D\left(t_0/2,x,z\right)\geq c_2 (R_1/R_2)^{2+d}\left(\frac{V(\delta_D(x))}{V(R_1)}\wedge 1\right)p_D\left(t_0/2,x_0,z\right).$$}
This  and the semigroup property imply
\begin{eqnarray*}
p_D(t,x,y)&=&\int \int p_D\left(t_0/2,x,z\right)p_D\left(t-t_0,z,w\right)p_D\left(t_0/2,w,y\right)dzdw\\
&\geq& c_2^2{ \(\frac{R_1}{R_2}\)^{4+2d}}l(x,y) \int \int p_D\left(t_0/2,x_0,z\right)p_D\left(t-t_0,z,w\right)p_D\left(t_0/2,w,y_0\right)dzdw\\
&=&c_2^2 { \(\frac{R_1}{R_2}\)^{4+2d}} l(x,y)p_D(t, x_0,y_0)\geq\frac{c_2^2}{2} { \(\frac{R_1}{R_2}\)^{4+2d}}l(x,y)p(t, x_0,y_0),
\end{eqnarray*}
where in the last step we used (\ref{LEB1}).
Since $p(t, x_0,y_0)=p(t,x,y)$, we obtain the conclusion.
\end{proof}
{We recall that the assumption $d>\la$ above yields the transience of the process $X$. We note that the results for recurrent unimodal L\'evy processes in dimension $1$ should be quite different:
for exterior domains in the case of  recurrent the isotropic stable L\'evy processes we refer to \cite{MR2722789}.

The following proposition may be proved in a similar way as \cite [Theorem 6.3] {BGR2}, where the result was shown for a complement of a closed ball. We leave the details to the reader.
\begin{prop}\label{Exit_time100} Let $\psi \in \WLSC{\la}{0}{\lC}\cap \WUSC{\ua}{0}{\uC}$ and $d>\ua$. Let $D$ be  a $C^{1,1}$ at scale $R_1$ and $D^c\subset \overline{B_{R_2}}$.
Then there are  constants $c_*=c_*(d,\psi), c^*=c^*(d,\psi)$ such that for all $x,y\in\Rd$ and $t>0$,
		$$c_* \(\frac{R_1}{R_2}\)^2\left(\frac{V(\delta_D(x))}{\sqrt{t}\wedge V(R_1)}\wedge1\right)\le  \p^x(\tau_{D}>t)\le c^* \left(\frac{V(\delta_D(x))}{\sqrt{t}\wedge V(R_1)}\wedge1\right).
 $$
\end{prop}
One can also prove sharp estimates of $\p^x(\tau_{D}>t)$ above by integrating the estimates in Theorem~\ref{exterior_lower}, but it results with a suboptimal dependence of comparability constants on $R_1/R_2$.

The following corollary is an immediate   consequence of Theorem~\ref{exterior_lower} and Proposition~\ref{Exit_time100}.
\begin{cor}\label{Exterior_Approx}
 Let $\psi \in \WLSC{\la}{0}{\lC}\cap \WUSC{\ua}{0}{\uC}$ and $d>\ua$.  Let $D$ be  a $C^{1,1}$
at scale $R_1$ and such that $D^c\subset \overline{B_{R_2}}$.
For all $x,y\in \Rd$ and $t>0$ we have
$$ p_{D}(t, x, y)
    \approx  \p^x(\tau_{D}>t) \p^y(\tau_{D}>t) p(t, x,y),$$
with comparability constant $C=C(d,\psi,R_2/R_1)$.
\end{cor}

The next  lemma is helpful to handle  halfspace-like $C^{1,1}$ sets.

{\begin{lem}\label{V_Product}Let $t_0>0$, $r_0=V^{-1}(\sqrt{t_0})$. Then, for $r>0$, $\lambda\geq 1$, $t>t_0$,
$$\frac{1}{\lambda+2}\(\frac{V(r)}{\sqrt{t_0}}\wedge 1\)\(\frac{V(r+\lambda r_0)}{\sqrt{t}}\wedge 1\)\leq\(\frac{V(r)}{\sqrt{t}}\wedge 1\)\leq \(\frac{V(r)}{\sqrt{t_0}}\wedge 1\)\(\frac{V(r+\lambda r_0)}{\sqrt{t}}\wedge 1\).$$
\end{lem}}
\begin{proof}
By subadditivity and monotonicity of $V$ we have  $$V(r_0\vee r)\leq V(r+\lambda r_0)\leq (\lambda+2)V(r_0\vee r).$$ Considering cases $r\le r_0$ and $r>r_0$, this observation easily leads to the conclusion.
\end{proof}
Here is our main result for halfspace-like $C^{1,1}$ sets. Recall that $\mathbb{H}_a$ is defined in Section \ref{sec:prel}.
\begin{thm}\label{halfspace-like}Let $\psi$ satisfy global  {\rm WLSC} and {\rm WUSC}, $D$ be $C^{1,1}$ at scale $R$ and $\mathbb{H}_a\subset D\subset \mathbb{H}_b$. Then for all
 $x,y\in \Rd$ and $t>0$, {
$$ p_{D}(t, x, y)
    \approx \p^x(\tau_D>t)\p^y(\tau_D>t)p(t,x,y)\qquad\text{and}\qquad \p^x(\tau_D>t)\approx\frac{V(\delta_D(x))}{\sqrt{t}}\wedge 1,$$}
and constants in the comparisons may be so chosen to  depend only on $d$, $\psi$, $a-b$ and $R$.
\end{thm}

\begin{proof}
Without loss of generality we may and do assume that $a>b=0$.
Let $x,y\in D$.
Due to Corollary \ref{lower_scaling}  and Theorem  \ref{heatKernelComplGlobal} it remains to prove   the  comparisons  for $t>t_0=c^*V^2(R)$,  where $c^*$ is the constant from  Corollary \ref{lower_scaling}.
Our arguments below are similar to those proving \cite[Theorem 1.2]{MR2776619}, where
where the result is proved for the isotropic stable L\'evy processes.
Let $r_0=V^{-1}(\sqrt{t_0})$, $\lambda=1+a/r_0$, $x_0=x+\lambda r_0 e_d$ and $y_0=y+\lambda r_0 e_d$, where $e_d=(0,\ldots,0,1)$.
By  Lemma \ref{freeChange} and \cite[Corollary 24]{2014-KB-TG-MR-jfa} the following comparison depends only on $d$, $\psi$ and $\lambda$:
$$p(t_0/2,x,z)\approx p(t_0/2,x_0,z), \quad x,z\in \Rd.$$
Since $\delta_D(x_0)\geq \delta_{\mathbb{H}_a}(x_0)\geq r_0$, we have $V(\delta_D(x_0))\geq\sqrt{t_0}$.
By Corollary \ref{lower_scaling} and  Theorem  \ref{heatKernelComplGlobal},  \begin{equation}\label{equiv}p_D(t_0/2,x,z)\approx \(1\wedge\frac{V(\delta_D(x))}{\sqrt{t_0}}\)p_D(t_0/2,x_0,z), \quad x,z\in \Rd,\end{equation}
where the comparability constant depends on dimension $\psi$, $a$ and $R$.
We denote $l(x,y)=\(1\wedge\frac{V(\delta_D(x))}{\sqrt{t_0}}\)\(1\wedge\frac{V(\delta_D(x))}{\sqrt{t_0}}\)$. By (\ref{equiv}) and the semigroup property,
\begin{eqnarray}p_D(t,x,y)&\approx&
l(x,y)\int_D\int_Dp_D(t_0/2,x_0,z)p_D(t-t_0,z,w)p_D(t_0/2,w,y_0)dzdw\nonumber\\
&=&l(x,y)p_D(t,x_0,y_0).\label{hsl1}
\end{eqnarray}
We have  $$\delta_D(x)+r_0\leq \delta_{\H_a}(x_0)\leq\delta_{\H_0}(x_0)\leq \delta_D(x)+2\lambda r_0,$$
hence, by Lemma \ref{V_Product},
$$\(1\wedge\frac{V(\delta_{\H_0}(x_0))}{\sqrt{t}}\)\(1\wedge\frac{V(\delta_D(x))}{\sqrt{t_0}}\)\leq (2+2\lambda)  \(1\wedge\frac{V(\delta_{D}(x))}{\sqrt{t}}\)$$
and
\begin{equation}\label{hsl2}\(1\wedge\frac{V(\delta_{\H_a}(x_0))}{\sqrt{t}}\)\(1\wedge\frac{V(\delta_D(x))}{\sqrt{t_0}}\)\geq   \(1\wedge\frac{V(\delta_{D}(x))}{\sqrt{t}}\).\end{equation}
The last two estimates also hold if $x_0, x$ are replaced by $y_0, y$.
We note that
$$p_{\mathbb{H}_a}(t,x_0,y_0)\leq p_D(t,x_0,y_0)\leq p_{\mathbb{H}_0}(t,x_0,y_0),$$ and $ p(t,x_0-y_0)= p(t,x-y)$. Also, $\delta_D(x_0)\approx \delta_{\H_a}(x_0)\approx\delta_{\H_0}(x_0)$ because $\delta_{\H_0}(x_0)\le \delta_{\H_0}(x_0)=a+\delta_{\H_a}(x_0)\le \lambda \delta_{\H_a}(x_0)$.
 From Corollary \ref{hk_Halfspace_Approx}, subadditivity of $V$,  and \eqref{hsl1} and \eqref{hsl2} (along with their variants for $y_0$ and $y$), we obtain
$$p_D(t,x,y)\approx \(\frac{V(\delta_D(x))}{\sqrt{t}}\wedge 1
\)\(\frac{V(\delta_D(y))}{\sqrt{t}}\wedge 1\)p(t, x, y).$$
This gives a sharp approximate factorization of $p_D$.
One consequence is that
$$\p^x(\tau_D>t)\leq c  \(\frac{V(\delta_D(x))}{\sqrt{t}}\wedge 1
\)\int_D p(t,x,y)dy\leq c  \(\frac{V(\delta_D(x))}{\sqrt{t}}\wedge 1
\).$$
By {Lemma \ref{lem:spgenC11}}
 a matching lower bound
holds for $0<t\leq t_0$.
If $t>t_0$, then by (\ref{equiv}) and the semigroup property,
$$p_D(t,x,y)\approx  \(\frac{V(\delta_D(x))}{\sqrt{t_0}}\wedge 1
\) p_D(t,x_0,y),$$
cf. the proof \eqref{hsl1}.
We integrate the comparison against $y$ and use \cite[Theorem 3.1]{KMR}, to get
\begin{align*}
\p^{x}(\tau_D>t)&\approx \(\frac{V(\delta_D(x))}{\sqrt{t_0}}\wedge 1
\)\p^{x_0}(\tau_D>t)\geq \(\frac{V(\delta_D(x))}{\sqrt{t_0}}\wedge 1
\)\p^{x_0}(\tau_{H_a}>t)\\
&\approx \(\frac{V(\delta_D(x))}{\sqrt{t_0}}\wedge 1
\)\(1\wedge\frac{V(\delta_{\H_a}(x_0))}{\sqrt{t}}\).
\end{align*}
We end the proof by using \eqref{hsl2}.

\end{proof}

\section{Examples}\label{sec:examples}

In a recent work \cite{2013arXiv1303.6449C}, Chen, Kim and Song provide estimates of Dirichlet heat kernels for a class of pure-jump Markov processes with intensity of jumps comparable to that
of a complete subordinate Brownian motion
with scaling.
In fact the assumptions of \cite{2013arXiv1303.6449C} imply the (scale invariant) boundary Harnack inequality, which leads to
the ``Lipschitz setting'' mentioned in the Section~\ref{sec:mot}, and allows to handle the so-called $\kappa$-fat sets (see \cite{MR2722789} for the case of the isotropic stable L\'evy processes).
In this sense \cite{2013arXiv1303.6449C} is a culmination of the line of research presented in \cite{MR2677618, MR2923420, MR2981852, 2013arXiv1303.6449C}.

Therefore in the examples below we focus on
processes which are not covered by \cite{2013arXiv1303.6449C}.
Namely, in  the first three examples the (scale invariant)  boundary Harnack inequality is not known or simply fails for some $C^{1,1}$  sets, but our method provides satisfactory estimates.
Our last two examples are more straightforward, and the reader may find others in \cite{2014-KB-TG-MR-jfa} and  \cite{BGR2}.

\begin{exmp}\label{ex:tru}
Let $\nu(x)={\log^\beta (1+|x|^{-1})}|x|^{-d-\alpha}\ind_{B_1}(x)$, $\alpha\in(0,2)$, $\beta\geq 0$.
It is known that the scale invariant boundary Harnack inequality fails for some $C^{1,1}$ sets for the corresponding
{\it truncated} L\'evy process
\cite{MR2391246}, but the characteristic exponent $\psi$  satisfies the desired scaling conditions. Indeed, it is easy to verify that $h(r)\approx r^{-2}\wedge \[\log^\beta (2+r^{-1})r^{-\alpha}\]$. Then, by
Lemma~\ref{ch1V}, $\psi(x)\approx h(|x|^{-1})\approx |x|^2\wedge \[\log^\beta (2+|x|)|x|^\alpha\]$ and $\psi\in  \WUSC{\alpha+\varepsilon}{1}{\uC_{\varepsilon}} \cap \WLSC{\alpha}{1}{\lC}$, if $\alpha<\varepsilon+\alpha<2$. Our results apply to this case,
e.g.
for $0<r<1/2$, $0<t<r^\alpha/|\log r|^\beta$ and  $x,y\in B_r$, we have
\begin{equation*}
p_{B_r}(t,x,y)\approx  \(1\wedge\frac{(r-|x|)^{\alpha}}{t\log^{\beta}\frac{1}{r-|x|}}\)^{1/2} \(1\wedge\frac{(r-|y|)^{\alpha}}{t\log^{\beta}\frac{1}{r-|y|}}\)^{1/2}\[\(t\log^\beta\frac{1}{t}\)^{-d/\alpha}\wedge\frac{t\log^{\beta}\frac{1}{ |x-y|}}{|x-y|^{d+\alpha}}\],
\end{equation*}
and the comparability constant depends only on $d$ and $\psi$.
\end{exmp}
\begin{exmp}
Let $T$ be a subordinator with L\'{e}vy density $\mu(r)=r^{-1-\alpha/2}\ind_{(0,1)}(r)$, $\alpha\in(0,2)$, and $X$  be a subordinate Brownian motion governed by $T$. Then $\psi(x)\approx |x|^2\wedge |x|^{\alpha}$ and $\nu(x)\approx e^{-|x|^2/4}|x|^{-2}$, $|x|\geq 1$. This $\psi$ satisfies WLSC and WUSC with $\la=\ua=\alpha$, but
the scale invariant boundary Harnack inequality does not hold (see \cite[Example 5.14]{2012arXiv1207.3160B}).
\end{exmp}

\begin{exmp}
Let  $\phi$  be a complete Bernstein function \cite{MR2978140} and  $\phi(|\cdot|^2)\in \textrm{WUSC}(\ua,0,\uC) \cap \textrm{WLSC}(\la,0,\lC)$. If $\psi(x)= |x|^2+ \phi(|x|^2)$ then, by  \cite[Proposition 4.5 and Theorem 4.4]{KMR} the renewal function $V$ of the ascending ladder-height process is a Bernstein function and   $\psi_1(x)=V(|x|^2)\in \WUSC{\ua}{1}{{\uC}_1} \cap \WLSC{\la}{1}{{\lC}_1}$ is the characteristic exponent of a subordinate Brownian motion. For this process it is not clear if the boundary Harnack inequality holds.
In particular, it is not clear how to construct a complete subordinate Brownian motion with comparable L\'evy measure.
Nevertheless, our approach applies because of scaling and isotropy.
\end{exmp}

In the next two examples we assume
global scaling conditions and we focus on estimates for exterior $C^{1,1}$  sets for the full range of time and space.
To the best of our knowledge such estimates were known only for the isotropic stable L\'evy process. Even the estimate from the next example seems to be new.
\begin{exmp}
Let $0<\alpha_1\leq \alpha_2<2$, $d>\alpha_2$ and $\psi(x)=|x|^{\alpha_1}+|x|^{\alpha_2}$. Then $\psi\in \WLSC{\alpha_1}{0}{1}\cap \WUSC{\alpha_2}{0}{1}$. In particular, by Lemma~\ref{densityApprox},
$$p_t(x)\approx (t^{-1/\alpha_1}+t^{-1/\alpha_2})^d\wedge\frac{t(|x|^{-\alpha_1}+|x|^{-\alpha_2})}{|x|^d},\qquad t>0, \ x \in \Rd,$$
and by Corollary \ref{Exterior_Approx},
$$p_{ \overline{B}_{r}^c}(t,x,y)\approx \(1\wedge\frac{(|x|-r)^{\alpha_1}\wedge(|x|-r)^{\alpha_2}}{t\wedge r^{\alpha_1}\wedge r^{\alpha_2}}\)^{1/2}\(1\wedge\frac{(|y|-r)^{\alpha_1}\wedge(|y|-r)^{\alpha_2}}{t\wedge r^{\alpha_1}\wedge r^{\alpha_2}}\)^{1/2}p(t,x,y),$$
where $r>0$, $t>0$, $x,y\in B_{r}^c$
and the comparability constant depends only on $d$, $\alpha_1$ and $\alpha_2$.
\end{exmp}
\begin{exmp}
If $f\in \WLSC{\la}{0}{\lC}\cap \WUSC{\ua}{0}{\uC}$ is nonincreasing and $\nu(|x|)=f(1/|x|)/|x|^d$, then $\psi$ has both global
scalings
\cite[Proposition 28]{2014-KB-TG-MR-jfa}. E.g. we may let $\alpha_1,\alpha_2\in (0,2)$ and
\begin{align*}
f(r)&=(r^{-1}\log(r+1/r))^{\alpha_1},\quad \mbox{or} \\
f(r)&=(r\log(r+1/r))^{-\alpha_1},\quad \mbox{or}\\
f(r)&=\left\{\begin{array}{ll}
r^{-\alpha_1} & \text{ if }0<r\leq 1,\\
r^{-\alpha_2}/2& \text{ if } r>1.
\end{array}
\right.
\end{align*}
In particular we do not need continuity of $\nu$, which was assumed in \cite{2013arXiv1303.6449C}.
The results from Section~\ref{sec:ud} give estimates which are uniform in the whole of time and space for exterior $C^{1,1}$ sets and halfspace-like sets.
For clarity, the global scaling conditions imply the scale boundary Harnack inequality (\cite{BHP_KSV2012} and \cite[Corollary 27]{2014-KB-TG-MR-jfa}), so short time estimates would follow from  \cite{2013arXiv1303.6449C} and \cite{BGR2},
if we also assumed continuity of the L\'evy density.

\end{exmp}

We finally suggest a possible generalization of our estimates which
relaxes
the assumption of monotonicity of the L\'evy density.

\begin{rem}\label{rem:gLc}
We can work with more general  isotropic  pure-jump L\'{e}vy processes.
Assume that
the L\'{e}vy measure of $X$ is absolutely continuous and its density function satisfies $\nu(x)=\nu_0(|x|)\approx f(1/|x|)/|x|^d$, where $f$ is nonincreasing and satisfies WLSC($\la,\theta,\lC$) and WUSC ($\ua,\theta,\uC$).
Then by \cite[Proposition 28]{2014-KB-TG-MR-jfa}, $\psi(x)\approx f(x)$ for $|x|\geq \theta$, hence $\psi$ satisfies (local) WLSC and WUSC.   By  \cite{MR2524930} for $\theta>0$ and \cite{MR2357678} for $\theta=0$ we get estimates for the heat kernel. This implies that $x\to p_t(x)$ is radial and almost decreasing locally in time and space (for all $t>0$ and $x\in\Rd$, if $\theta=0$). Moreover,  the scale invariant Harnack inequality { holds \cite{MR2524930, MR2357678}}. This allows to repeat the arguments in Section~\ref{sec:uni}, Section~\ref{sec:ud} and Section~\ref{sec:examples}, to obtain analogous estimates for bounded $C^{1,1}$ open sets if $\theta>0$ and for all considered $C^{1,1}$ sets if $\theta=0$.

\end{rem}


\end{document}